\documentclass[11pt]{book}              
\raggedright                            

\usepackage{graphicx}
\usepackage{amsmath}
\usepackage{amsfonts}
\usepackage{booktabs}

\title{\bf  Lecture Notes on \\  Numerical Smoothness and Error Analysis}    
\author{ Tong Sun}              
\date{\today}                           

\newtheorem{theorem}{Theorem}[section]
\newtheorem{lemma}[theorem]{Lemma}

\newtheorem{definition}[theorem]{Definition}

\numberwithin{equation}{section}

\def\P{\mathcal P}

\begin{document}                   	
\frontmatter                          	  
\maketitle                            	  
\tableofcontents    		  
\listoffigures                  		  
\mainmatter                           	  

\chapter{Numerical smoothness}                

This lecture notes is to introduce the idea and framework of \texttt{numerical smoothness} to the audience. The goal of the discussions is to show that numerical smoothness provides a fundamental framework to do error analysis for numerical solutions of time-dependent ordinary and partial differential equations, or evolution equations. Unlike the Lax equivalence theorem focusing on numerical stability, where one encounters many difficulties with nonlinear equations solved with complex numerical schemes, the numerical smoothness framework can naturally handle nonlinear multi-physics equations solved by complex numerical schemes. 

We are going to show that, given a well-posed time-dependent differential equation and a consistent numerical scheme, one can do the following. 

(A)  A set of smoothness indicators can be designed for the numerical solution, denoted as $S^p$ and $T^k$, where $S^p$ shows whether the numerical solution is {\it numerically smooth} in space, $T^k$ shows whether the numerical solution is {\it numerically smooth} in time. \\
(B) It can be proven that the boundedness of the smoothness indicators is a necessary condition for the numerical solution to converge to any smooth function in the desired optimal convergence rate of the scheme. If a numerical solution is not converging to any smooth function, it is certainly not converging to a smooth solution of an evolution equation.\\
(C) It can be proven that the numerical solution is within an error estimate of the form
 \begin{equation} \label{estimate}
h^{p+1} F(S^p) +  \tau^k G(S^p,T^{k})
\end{equation}
from the true solution, where $h$ is the mesh size, $\tau$ is the time step size, $p$ and $k$ are the desired convergence orders, $F$ and $G$ are computable functions of $S$ and $T$.  \\
(D) The functions $F$ and $G$ also contain the estimation on how the error is propagated. The propagation rate is how error should be propagated by the evolution equation, which is obviously at least as good as the error propagation of any convergent numerical scheme.  

Here are a few remarks on the significance of the above:

(a) If one can prove the boundedness of $S^p$ and $T^{k}$ before the computation, the error estimate is {\it a priori}.\\
(b) One can always compute the indicators $S^p$ and $T^k$, hence always has an {\it a posteriori} error estimate.\\
(c) If $S^p$ or $T^{k}$ are proven or computed to be too large, the optimal convergence rate must have been lost.\\
(d) Thus, the numerical smoothness framework has {\it a priori} error estimate, {\it a posteriori} error estimate and failure detection done in a unified theoretical analysis. 

Now the audience may have a lot of doubts or questions, such as how the smoothness indicators are defined, how the so called numerical smoothness comes into play,  or even what it means by saying that a numerical solution is numerically smooth. The answers to these questions need to start from the two alternatives ways of error splitting, one leading to numerical stability in the numerical stability framework, the other one leading to the innovative numerical smoothness concept and the numerical smoothness framework.

These lecture notes are based mostly on the author's work listed among the bibliography references. 

\section{Numerical stability vs. numerical smoothness}

The evolution equations we consider can be represented abstractly, using Banach space terminology, as in the textbook of Richtmyer and Morton \cite{morton}. We wish to find a solution of the evolution equation 
\begin{equation} \label{DE}
u_t = f(u), \qquad t \in [0,T],
\end{equation}
with the initial condition
\begin{equation} \label{IC}
u(0) = u_0,
\end{equation}
where $f$ is a nonlinear operator from the Banach space $\cal{B}$ containing $u$ to a proper range space. For more details of the abstract problem, one can go to Chapter 3 of the reference \cite{morton}, where a linear problem is presented. Here we will not specify the details about the well-posedness of our abstract problem. However, one can find more details in the concrete cases from \cite{sun}, \cite{tsEwing}, \cite{sunF}, \cite{Rumsey}, \cite{Fode} and \cite{WENO}. We choose to consider an autonomous problem in order to avoid tedious notations. 

A numerical scheme solving (\ref{DE}) can be abstractly written as
\begin{equation} \label{scheme}
u_N(t_{n+1}) = u_N(t_n) + \tau\, \Phi(u_N(t_n)), 
\end{equation}
where the index N stands for numerical soluton. The index $n$ is the step number in the discretization of the time domain $[0,T]$,  where  $\tau = T/K$ , $K$ is the total number of time steps, so the time domain is partitioned by $t_0=0$, $t_1=\tau$, $\cdots$, $t_{n+1}=t_n + \tau$, $\cdots$, $t_K=K \tau = T$. $\Phi$ is the function (rule, formulas) of computing $u_N(t_{n+1})$ from $u_N(t_n)$, which could be a Runge-Kutta scheme, a WENO scheme for a conservation law, or any single step scheme for solving (\ref{DE}). The spatial discretization are hiding in the details of the $\Phi$ for now. 

The global error of the numerical solution $u_N(t_n), n=1, \cdots, K$ is
$$
e_{n} = u(t_n)-u_N(t_n), \qquad n=1, \cdots, K.
$$
After $e_n$ has been estimated and $u_N(t_{n+1})$ has been computed, to estimate $e_{n+1}$, one needs to split it to a local error occuring in the step and a propagated error of $e_n$. There are two fundamentally different ways of error splitting, which are figured as in Splitting 1 and Splitting 2. 

\vspace{0.08in}

\setlength{\unitlength}{0.8mm}
\begin{picture}(120,40)(-10,10)
	\linethickness{0.6pt}
	\qbezier(12,23)(27,29)(42,38)
	\linethickness{0.5pt}
	\qbezier(12,23)(27,28.5)(42,34)
	\qbezier(12,17)(27,21)(42,25)
	\put(5,15){\makebox(0,0)[b]{$u_N(t_n)$}}
	\put(5,21){\makebox(0,0)[b]{$u(t_n)$}}
	\put(51.5,23){\makebox(0,0)[b]{$u_N(t_{n+1})$}}
	\put(51.5,32){\makebox(0,0)[b]{$\hat{u}_N(t_{n+1})$}}
	\put(50,37){\makebox(0,0)[b]{$u(t_{n+1})$}}
	\put(-3, 10){\vector(1,0){60}}
	\put(-2,  9){\vector(0,1){35}}
	\put(11.8,10){\line(0,1){13}}
	\put(42,10){\line(0,1){28}}
	\put(12,6){\makebox(0,0)[b]{$t_n$}}
	\put(42,6){\makebox(0,0)[b]{$t_{n+1}$}}
	\put(25,0){\makebox(0,0){Splitting 1. Numerical Stability}}

	\linethickness{0.6pt}
	\qbezier(87,23)(102,29)(117,38)
	\linethickness{0.5pt}
	\qbezier(87,17)(102,21.5)(117,29)
	\qbezier(87,17)(102,21)(117,25)
	\put(80,15){\makebox(0,0)[b]{$u_N(t_n)$}}
	\put(80,21){\makebox(0,0)[b]{$u(t_n)$}}
	\put(126.5,23){\makebox(0,0)[b]{$u_N(t_{n+1})$}}
	\put(125,28){\makebox(0,0)[b]{$\tilde{u}(t_{n+1})$}}
	\put(125,37){\makebox(0,0)[b]{$u(t_{n+1})$}}
	\put(71, 10){\vector(1,0){60}}
	\put(72,  9){\vector(0,1){35}}
	\put(86.8,10){\line(0,1){13}}
	\put(117,10){\line(0,1){28}}
	\put(87,6){\makebox(0,0)[b]{$t_n$}}
	\put(117,6){\makebox(0,0)[b]{$t_{n+1}$}}
	\put(105,0){\makebox(0,0){Splitting 2. Numerical Smoothness}}
\end{picture}

\vspace{0.55in}

Here, in Splitting 1, $\hat{u}_N(t_{n+1})$ is defined by
 \begin{equation} \label{hat}
\hat{u}_N(t_{n+1}) = u(t_n) + \tau\, \Phi(u(t_n)).
\end{equation}

In Splitting 2, $\tilde{u}(t)$ is defined in each time step $[t_n,t_{n+1}]$ by
\begin{equation} \label{tilde}
\tilde{u}_t = f(\tilde{u}), \qquad t \in [t_n,t_{n+1}],
\end{equation}
with the initial condition
\begin{equation} \label{ICtilde}
\tilde{u}(t_n) = u_N(t_n).
\end{equation}

{\bf Remark:} Both of these splittings appeared in the more specialized numerical ODE books, such as \cite{butcher}, \cite{Gear} and \cite{wanner}, but the authors did not go very far along Splitting 2. In most numerical analysis textbooks, we often see authors using Splitting 1.

{ \bf A. Numerical stability along with ideal local error}

As seen in (\ref{hat}), $\hat{u}_N(t_{n+1})$ is the non-computable numerical solution which would be obtained by using the same scheme $\Phi$, in the step $[t_n,t_{n+1}]$, with initial value $u(t_n)$.  After rewriting Splitting 1 in a formula, the total error at $t_{n+1}$ is split into a local error and a propagation error by $\hat{u}_N(t_{n+1})$ as
\begin{equation} \label{numPropa}
u(t_{n+1})-u_N(t_{n+1}) = \underbrace{u(t_{n+1})-\hat{u}_N(t_{n+1})} +\underbrace{\hat{u}_N(t_{n+1})-u_N(t_{n+1})}.
\end{equation}

In most textbooks, a typical way of defining local error is to put the real solution $u(t)$ in a scheme to find the residual $\hat{\xi}_{n+1}$, which is  equivalent to looking at $u(t_{n+1})-\hat{u}_N(t_{n+1})$.  In fact, let
\begin{equation}
u(t_{n+1}) = u(t_n) + \tau\, \Phi(u(t_n)) + \hat{\xi}_{n+1},      \label{idealLocal}
\end{equation}
where $\hat{\xi}_{n+1}$ is the residual, telling us by how much the exact solution $u(t)$ does not satisfy the scheme. Subtracting this equation (\ref{idealLocal}) by (\ref{hat}),  we get 
$$
\hat{\xi}_{n+1} = u(t_{n+1})-\hat{u}_N(t_{n+1}) =u(t_{n+1}) - u(t_n) - \tau\, \Phi(u(t_n))  .
$$
The residual $\hat{\xi}_{n+1}$ is defined to be the local error in Splitting 1. In the form $u(t_{n+1}) - u(t_n) - \tau\, \Phi(u(t_n))$ at the very right hand side of the last equation, one can use Taylor expansion or other tools to convert the local error to the form of 
\begin{equation} \label{localErrorS1}
\| \hat{\xi}^{n+1} \| \le C \tau^{k+1} \max_{t \in [t_n,t_{n+1}]} \{ \|u^{(k+1)}(t) \| \},
\end{equation}
 where $k \ge 1$ depends on the scheme $\Phi$.  

The reason of defining $u(t_{n+1})-\hat{u}_N(t_{n+1})$ to be the local error is quite obvious: the smoothness of $u(t)$ can be used only in this way to estimate the local error.

As a consequence of this convenient local error definition, one has to face the second part of the split error, namely, $\hat{u}_N(t_{n+1})-u_N(t_{n+1})$, which is clearly the numerical propagation of the error $e_n = u(t_n)-u_N(t_n)$ by the scheme. Thus, Splitting 1 leads to the numerical error propagation, also considered as numerical stability in the form of the distance of two numerical solutions. 

When people use numerical stability in proving the convergence of the scheme, they usually try to prove that, in an appropriate norm, 
\begin{equation} \label{numSta}
\| \hat{u}_N(t_{n+1})-u_N(t_{n+1}) \| \le (1+\beta \tau) \| u(t_n)-u_N(t_n) \| \le e^{\beta \tau}  \| u(t_n)-u_N(t_n) \| , 
\end{equation}
where $\beta$ can be the Lipschitz constant of $\Phi$, depending on the scheme. In most textbooks and research papers, people combine (\ref{numPropa}), (\ref{localErrorS1}), and (\ref{numSta}) to prove their main theorem on the global error estimate, which typically looks like
$$
\| u(T) - u_N(T) \| \le C \max_{t \in [0,T]}  \{ \| u^{(k+1)}(t) \| \}  e^{\beta T} \tau^{k}. 
$$

It would be ideal to have $\beta < 0$, with which one can prove a uniform convergence of the numerical solution in the entire time domain. The analysis to take advantage of $\beta<0$ is related to the one-sided Lipschitz condition and the B-stability concept, which only works for very few very simple schemes of low order. For nonlinear ODE systems, the distance between two solutions is often expected to have exponential growth for short period of time. In such cases, $\beta < 0$ is impossible. Hence more sophisticated technical research is needed.

If $\beta = 0$, there will be a simple accumulation of local error into global error. This is often related to the conservation of certain physics quantities. For example, the mass conservation of the scaler hyperbolic conservation laws resulted to the $L_1$-contraction. It is desirable in this case to look for a scheme which also possesses a $L_1$-contraction property between two numerical solutions, but all the modern schemes like discontinuous Galerkin finite element method or WENO method cannot be proven to have $L_1$-contraction. Instead, people often settle down with $\beta>0$ in their proof of numerical stability. 

For some strongly nonlinear dynamical systems, there exist short time exponential growth in certain components of the solutions of the system. In a case like this, it is impossible to expect $\beta \le 0$ in a Lipschitz condition or one-sided Lipschitz condition. For such systems, one has to deal with long term error analysis by other methods. There is very limited research on this topic.

Another thing ruining $\beta \le 0$ is the choice of the norn used in the analysis of nonlinear evolution equations. Using the scalar conservations law again as an example, we know that the distance between two solutions of the evolution equation is preserved in the $L_1$-norm, but it is to grow to the height of the shock in $L_{\infty}$-norm. While using a finite element method, people tend to choose $L_2$-norm for its convenience in the finite element formulation. However, near a location of shock formation, the $L_2$-norm of the distance between two solutions is doomed to have exponential growth. 

Unfortunately, away from some simple equations of  ODEs or parabolic PDEs solved by a few simple schemes, one often ends up getting $\beta>0$. Both nonlinearity of the evolution equation and complexity of the numerical scheme will contribute to make $\beta$ larger. In the global error estimate, one will see a factor $e^{\beta T}$, if the numerical solution is to be computed in $t \in [0, T]$. 

In many cases, proving (\ref{numSta}) is simply too hard due to nonlinearity, scheme complexity, and other technical difficulties. For this reason, in many cases, although we can produce a numerical solution which looks perfect, we just cannot prove the convergence of the scheme, even more hopeless for having a practically useful error estimate.

For linear evolution equations, there is the well-known Lax Equivalence Theorem \cite{lax}. {\it Assuming a scheme is consistent with the equation, numerical stability is necessary and sufficient for convergence.} By careful reading into the details of the original proof of the Lax Equivalence Theorem, one can recognize that the sufficiency proof was in the framework of our Splitting 1. One can also read the textbook of Richtmyer and Morton, Section 3.5 of \cite{morton}.

{\bf B. Numerical smoothness along with exact error propagation}

In Splitting 2, the error is split in a different way. $u(t)$ and $u_N(t)$ are defined as before. For the entire time domain, $\tilde{u}(t)$ is defined piecewise, (\ref{tilde}) and (\ref{ICtilde}). Namely, in each time step $(t_n,t_{n+1}]$, $\tilde{u}(t)$ is the solution of the evolution equation with initial value $u_N(t_n)$. Here the total error at $t_{n+1}$ is split as
\begin{equation} \label{numSmo}
u(t_{n+1})-u_N(t_{n+1}) = \underbrace{u(t_{n+1})-\tilde{u}(t_{n+1})} + \underbrace{\tilde{u}(t_{n+1})-u_N(t_{n+1})}. 
\end{equation}
The first difference on the right hand side (RHS) is the propagation of $e_n = u(t_n)-u_N(t_n)$ by the evolution equation.
The second difference is the actual local error, which is the error of $u_N(t_{n+1})$ while approximating $\tilde{u}(t_{n+1})$.

The easy part  of Splitting 2 is its error propagation. The error propagation by the evolution equation (exact error propagation) is exactly how error should be propagated. Even if one does not know how the solutions of the differential equation propagate error, the bottom line is that the exact error propagation is at least as good as the numerical error propagation of any convergent scheme. One famous example is the proof of the $L_1$-contraction property of scalar conservation laws, the tool used for the proof was the Lax-Friedrichs scheme.  For ODE problems, the readers will see a case study in Chapter 3, where the advantage of the error propagation by evolution equations will be demonstrated with a very different technique.  

The other part of Splitting 2 is the local error $\xi_{n+1} = \tilde{u}(t_{n+1}) - u_N(t_{n+1})$. This is where the innovative concept of numerical smoothness comes into play. For readers who are not yet familiar with the idea of numerical smoothness for error analysis, the following explanations may be important.

Among the pioneers in this field, C.W. Gear was one of very few experts who used Splitting 2 in his textbook \cite{Gear}. After all, when the solution process has obtained $(t_n,u_N(t_n))$, the natural thing to do next is to approximate $\tilde{u}(t)$ to get a $(t_{n+1},u_N(t_{n+1}))$. Gear focused on ODE problems. For non-stiff ODEs, $\tilde{u}(t)$ is typically smooth in the classical sense, hence estimating the local error $\tilde{u}(t_{n+1})-u_N(t_{n+1})$ meets no trouble. For stiff ODEs, the Lipschitz constant $L$ is large. Gear tolerated the fact that the (k+1)-st derivative $\tilde{u}^{(k+1)}(t)$ is proportional to $L^{k+1}$ in most part of the domain. Because $L$ can be considered as a fixed constant, as long as the original problem is indeed an ODE, the error bound still implies convergence. 

In fact, only when the numerical scheme send the stiff components of the numerical solution into a small part of the domain, $\tilde{u}(t)$ can be small. The solutions of the stiff ODEs do have the property of doing this. Good numerical schemes should also have this property. This property is understood as numerical smoothing or smoothness maintenance. 
Splitting 2 provides the framework in which numerical analysts can monitor if the schemes have this property, as well as to use this property in local error control. 

A trivial example might help some readers here. Consider the model problem 
$$ w' = \lambda w, \qquad w(0) = 1, \qquad \lambda = -100. $$
On the surface, the high order derivative $w^{(k)}(t) = \lambda^k w(t)$, which looks huge. Digging slightly deeper, since $w^{(k+1)}(t) = \lambda w^{(k)}$, we get $w^{(k)} (t)= e^{\lambda t} w^{(k)} (0) = e^{\lambda t} \lambda^k $, which quickly decays toward $0$ through a short transient process. In fact, if we consider the 4th derivative, $w^{(4)}(0) = 100^4 = 10^8$, $w^{(4)}(0.1) \approx 4540$, $w^{(4)}(0.2) \approx 0.2062$, $w^{(4)}(0.3) \approx 0.000009358$. Obviously, the original solution of the equation goes into the narrow region of the domain where the smoothness of the solution improves quickly. One wants to choose a numerical scheme and an appropriate step size so that the $\tilde{u}(t)$ also appears in the narrow region. A numerically stable solution does not necessarily guarantee that this will happen.  

In another way of wording, a good numerical scheme should mimic a transient process for a stiff problem, and a corresponding good error analysis needs to take advantage of this. This kind of good schemes have correct numerical smoothing property. For example, the backward Euler method mimics the smoothing process much better than the forward Euler method with large step sizes. However, if we use tiny time step sizes, both the methods will do well.

In real life, a large chunk of ODEs to be numerically solved are the semi-discretization of hyperbolic or parabolic PDEs, which are severely stiff because the Lipschitz constant $L$ is proportional to $1/h$ or $1/h^2$, where $h$ is the mesh size of the spatial discretization, and $L$ goes to $\infty$ as $h$ goes to $0$.  What seems even worse is that the semi-discrete solutions for PDEs are piecewise polynomials, cell averages, or discrete nodal points, which are not smooth functions for sure. For example, piecewise polynomials have jumps or derivative jumps along cell boundaries all over the domain, consequently $\tilde{u}(t)$ can be full of shocks and rarefaction waves. If framed by these facts, one would believe that there is no smoothness to talk about for $\tilde{u}(t)$. Perhaps for this reason, the concept of numerical smoothness as we are talking about here did not appear earlier.

While the exact error propagation is an advantage for us to take, there is a new issue: the smoothness of $\tilde{u}(t)$ needs to be established, so that the local error $\tilde{u}(t_{n+1})-u_N(t_{n+1})$ can be estimated. Since $\tilde{u}(t)$ is a solution of the evolution equation, its smoothness relies on its initial value $u_N(t_n)$. Because $u_N(t_n)$ is the result of the numerical computation in the past steps before reaching $t_n$, we need that the scheme or schemes we have used so far have had certain {\bf numerical smoothing} effect or smoothness maintenance, so that $\tilde{u}(t)$ has certain smoothness. Once we have established the smoothness of $\tilde{u}(t)$, we can obtain a local error estimate for $\tilde{u}(t_{n+1})-u_N(t_{n+1})$ by using the consistency of the schemes.

The above approach of exact error propagation and numerical smoothness analysis has been used in \cite{sun}, \cite{tsEwing}, \cite{sunF}, \cite{Fode}, \cite{Rumsey} and \cite{WENO} for long time error analysis. In these papers, some examples of nonlinear problems and complex schemes are presented. There are also details about how to apply the idea to parabolic and hyperbolic  equations with various spatial discretization. Here we summarize the idea in the following simple paradigm, to point out that numerical stability can be replaced by numerical smoothness.  

{\bf Theoretical Paradigm}:  {\it Under the assumption of consistency, 

\hspace{1.13in}  Numerical Smoothness $\Leftrightarrow$ Optimal Error Estimate. }

This paradigm also applies to nonlinear evolution equations and complex schemes.

Remarks: For non-stiff ODEs, $\tilde{u}(t)$ should be smooth naturally. For stiff ODEs, $\tilde{u}(t)$ can be not so smooth if the scheme did not handle the stiffness well, see the following example. For PDEs, $\tilde{u}(t)$ is typically nonsmooth as piecewise polynomial approximation functions, thus the local error analysis can be very tedious, but doable because the local error analysis is usually local in space and time. Thus, parabolic problems are harder than hyperbolic problems here, because hyperbolic waves are purely local.   

In many textbooks, one can find explanations about the concept of numerical stabiliy. Typically, the understanding is that a stable scheme does not amplify small local truncation errors too badly, while a unstable scheme amplify small errors too much and ruin the numerical solution. When we work with numerical smoothness instead of numerical stability, we have a totally different  understanding. We choose the framework where error propagation is only done by the evolution equation. If we do not have enough numerical smoothness, we produce a large local error at each step. If we use an anti-smoothing scheme, we produce larger and larger local error step after step. Meanwhile, if we use a scheme with certain numerical smoothing property and adaptively choose time step sizes, we can control the local error.

\section{An illustrative example: numerical smoothness}

Figure \ref{figSTSM} shows an example which explains the understanding of the role of numerical smoothness. The equation is the ODE
$$ u' = 18 (\cos t - u)$$
in the time domain $[0, 2]$ with the initial condition $u(0)=0$. This is a typical negative feedback model with the control signal $\cos t$ to be followed by the solution $u(t)$. 
The coefficient $18$ makes the problem slightly stiff. The exact solution is 
$$ u(t) = \frac{324}{325} (\cos t-e^{-18 t}) + \frac{18}{325} \sin t, $$
or, in a more informative form,
$$ u(t) = \sqrt{\frac{324}{325}} \cos( t-\alpha) - \frac{324}{325} e^{-18 t}, $$ 
where $\alpha=\arctan\frac{1}{18}$ is the angle of the phase lag, $\sqrt{\frac{324}{325}}$ shows the small amplitude shrink from the amplitude $1$ of $\cos t$. The exponential term $ e^{-18 t}$ gives us the transient process, and the coefficient  $18$ tells us the strong contraction of $u(t)$ as a moving attractor \cite{tsEwing}.  

In the picture, the grey dash-dot curve is $\cos(t)$. The blue solid curve is the exact solution satisfying $u(0)=0$. The red dot curve is the explicit Euler solution with step size $\tau=0.1$. Obviously, with $\tau=0.1$, the Euler solution is stable but has insufficient smoothing, resulting in the oscillation, which makes the red solution unacceptable. More importantly, if we consider the DE solution $\tilde{u}(t)$ satisfying $\tilde{u}(t_k) = u_k$, where $u_k$ is the nodal value of the red solution at $t_k$, this DE solution does not mimic the smoothness of the original solution $u(t)$ until $k$ gets fairly large.  The green dash curve is the backward Euler solution also with $\tau=0.1$, obviously stable, and also provides sufficient smoothing, but has large local error in the first a few steps due to the low convergence rate, which in turn hurt the smoothing effect.

\begin{figure}[htbp]
    \centering
    \includegraphics[width=0.8\textwidth]{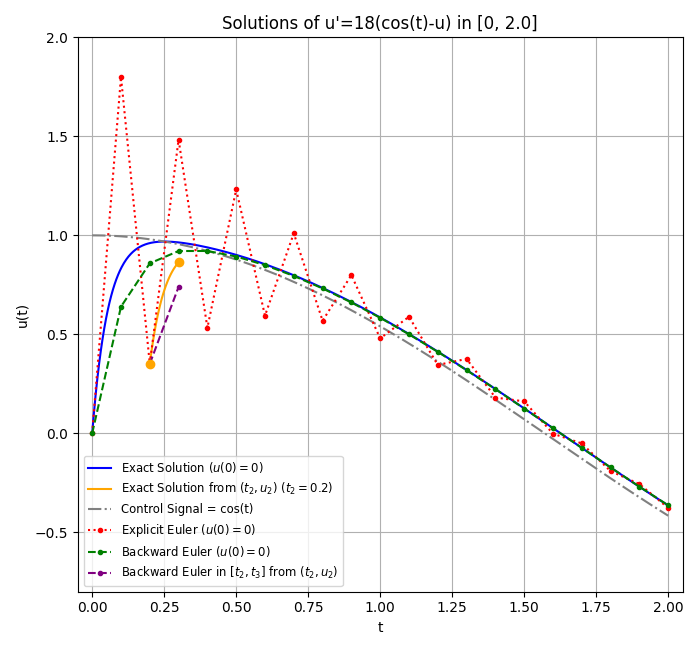} 
    \caption{Numerical Stability vs. Numerical Smoothness}
    \label{figSTSM}
\end{figure}

In the Lax framework, the behavior of the explicit Euler solution is usually said to be already stable when $\tau$ is reduced to $0.1$, but not converging yet. As for the understanding of the behavior under the numerical smoothness framework, let us focus on the step $[t_2,t_3]$. The solid orange curve is the ODE's solution $\tilde{u}(t)$ with the initial point marked by the big orange dot $(t_2, u_2)$, where $u_2$ is the solution of the explicit Euler scheme at $t_2$. In the time step $[t_2,t_3]$, the explicit Euler solution $(t_3,u_3)$ (in red) still has a huge local error from the orange curve. The backward Euler solution (in purple) also has a pretty bad local error from the orange curve, even though it is much better than the explict solution. Here, instead of blaming the stability of the explicit Euler, we blame the insufficient smoothing of the explicit Euler which made the $\tilde{u}(t)$ not as smooth as $u(t)$ in $[t_2,t_3]$. After all, the local error defined in Splitting 2 for either scheme is the distance from  the numerical solution to the orange curve. Since the previous steps of the explicit Euler have given us the orange curve $\tilde{u}(t)$ in $[t_2,t_3]$ which is not very smooth, there is no way to make the local error small unless we adaptively adjust the time step size $\tau$. Based on the above understanding, the so called numerical smoothness is about the smoothness of the $\tilde{u}(t)$ in each step.

To summarize the idea in other words, numerical smoothness is to see, for each $n$, whether a scheme delivers a $u_n=u_N(t_n)$ which makes the DE solution $\tilde{u}(t)$ smooth enough in $[t_n,t_{n+1}]$. If $\tilde{u}(t)$ is not smooth enough, the numerical solution is already ruined. If $\tilde{u}(t)$ is numerically smooth enough in $[t_n,t_{n+1}]$ as the boundedness of the smoothness indicator has told us, we just need to choose the scheme for $[t_n,t_{n+1}]$ and make sure $u_{n+1}=u_N(t_{n+1})$ will be as good as $u_n$ in the following two sense: error control and numerical smoothness. Instead of the two design goals for schemes in the Lax framework, consistency and numerical stability, our two design goals are consistency and numerical smoothness. When the smoothness indicators inform us a failure, we should either adjust the time step size or modify the scheme.

\section{Advantages of numerical smoothness over stability}

There are several advantages of using numerical smoothness over numerical stability in error analysis. The fundamental difference between numerical smoothness and numerical stability is the following. To be used in error propagation analysis, numerical stability is about the relation between two numerical solutions, one of them is always not computable, hence there is no natural stability indicator.  On the contrary, numerical smoothness is only about one numerical solution, namely the computed numerical solution, so smoothness indicators are always available. The impact of this difference is far and wide reaching.

1. Since only the computed numerical solution is involved, the loss of superposition in a nonlinear problem will not block us from computing a smoothness indicator. Hence numerical smoothness analysis works better for nonlinear problems.

2. No matter what numerical schemes are used for time-discretization, the smoothness indicator at each $t_n$ is computed in the same way based only on  the spatial discretization. Hence an error estimate can be done regardless of scheme complications in time. For ODEs, the smoothness indicator only depends on the equation, not the scheme. 

3. Error propagation is estimated by using evolution equations instead of schemes. All the potential troubles of numerical error propagation analysis are avoided. After all, error propagation by evolution equation is exactly how error should be propagated. Without complication of a numerical scheme, one can apply any contraction properties of the ODE and PDE solutions, including the attraction property of a solution, a limit cycle, a strange (chaotic) attractor, the $L^1$-contraction of the conservation laws, or any other attractive objects useful for error analysis. This is actually what makes long time error estimation possible \cite{tsEwing} \cite{sunF} \cite{Rumsey} \cite{WENO}.

4. Adaptive time-stepping is natural, since the smoothness of $\tilde{u}(t)$ is revealed by the smoothness indicator in each step $[t_n, t_{n+1}]$ \cite{sun} at the beginning of the time step. Within a step, the local error estimation is {\it a priori}.

5. The additional cost on computing the smoothness indicator is almost negligible. As shown in \cite{sun}, \cite{tsEwing}, \cite{sunF}, \cite{Rumsey}, \cite{Fode} and \cite{WENO}, computing the smoothness indicator involves a few actions of a sparse stiffness matrix and a mass matrix inversion. The stiffness matrices are computed anyway, while the mass matrix inversion can be avoided by mass lumping.  

6. When $u_N(t_n)$ has been obtained, according to our way of error analysis for the next step to be done by Splitting 2, we focus on (1) the exact error propagation and (2) the estimation of the real local error relying on the numerical smoothness indicator. One can see how much freedom we have in making the choice of the scheme(s) to be used in $[t_n,t_{n+1}]$. We can choose a different time discretization scheme for every step as long as we can control the local error as desired. We prefer to keep the spatial mesh for efficiency, but we can adjust the mesh whenever necessary as long as the reconstruction of $u_N(t_n)$ to fit the new mesh does not contribute too much (big) error toward the local error of the new time step. There is no need to worry about how any modification of the scheme could ruin our numerical stability. 

7. The smoothness indicators play the role of a two-fold sword. Checking whether the scheme(s) used before $t_n$ has maintained smoothness or delivered smoothing; and providing the current smoothness we have to use for controlling the local error.  

8. From the case study of Chapter 3, one can see that, under the numerical smoothness framework, we can even estimate error propagation with localized analysis. This is very important for nonlinear systems, especially when there is short time exponential growth in a nonlinear problem. 

With the list of all the advantages, one can see the significance of the innovative concept of numerical smoothness.

\section{Considerations on numerical stability}

The concept of numerical stability has a long history, starting from 1928 when Courant, Friedrichs and Levy proved the CFL necessity condition, which points out that the discrete domain of dependence must cover the continuous domain of dependence for finite difference schemes solving hyperbolic PDEs. This has been considered as the earliest stability condition, although it is not about the continuous reliance of a solution on initial value but about if all the information affecting the solution is approximately taking into account. 

In the 1940s, von Neumann developed the Fourier analysis method for analyzing the error of a finite difference solutions of linear parabolic and hyperbolic PDEs. This method allows people to observe what happens with the error in the frequency domain, in addition to space-time domain. It helps people identifying what is going wrong at the high frequencies, where the numerical solutions usually present symptoms of fast exponential growth first. These symptoms are called numerical instability, although the right understanding should be numerical anti-smoothing. Nevertheless, the necessary conditions of the form $\Delta t \le C h^k$ with $k = 1$ or $2$ derived from the von Neumann analysis are indeed necessary, sometimes even inadequate for the needed smoothing of parabolic equations. 

In 1956, Lax and Richtmyer establish the well-known Lax Equivalence Theorem: for a well-posed linear initial-value problem, 
$$\text{Consistency} + \text{Stability} \iff \text{Convergence}.$$ 
This has been unofficially referred as the fundamental theorem of numerical analysis, although the theorem is limited to linear time-dependent differential equations solved by linear numerical schemes. The result of the equivalence theorem has been considered as applicable in all branches of numerical analysis, although the meaning of stability is different in different branches. For example, the loss or weakening of coerciveness or inf-sup condition in an elliptic PDE problem is a well-posedness issue, very different from the issue of continuous reliance on initial value in time-dependent differential equations. Mixing all error sensitivity issues under the name of stability or instability do not help us.

Frankly speaking, the Equivalence Theorem may have done more harm then help over the years. As an equivalence statement,  the necessity is inadequate, the sufficiency is limited to linearity. Meanwhile, most people were believing that the equivalence means numerical stability of schemes can never be bypassed. A second thought won't even get started. People simply believe that a consistent scheme must be stable and only need to be stable, and the scheme is certainly responsible for error propagation. Thus, we are trapped in the problematic equivalence.

In the literature of numerical solutions of evolution equations, numerical stability has dozens of different definitions. One can divide them into two categories, (A) about a single numerical solution; and (B) about the difference between a pair of numerical solutions. In category (A), we have the Lax stability as in the Lax Equivalence Theorem, A stability, AN-stability, A($\alpha$)-stability, TVD stability, and so on. In category (B), we have algebraic stability, B-stability, and so on. 

For a linear evolution equation $u' = Au$ and a numerical scheme $u_{n+1} = C_{\tau} u_n$, Lax defined the numerical stability of the scheme by the uniform boundedness of $\| C_{\tau}^n \|$ for all $\tau \le \delta$ for a $\delta>0$ and $n \tau =T$.  This is about one numerical solution, and only involve the numerical solution itself, not involving the derivatives of the numerical solution of any order. Obviously, Lax stability belongs to category A. Some other stability definitions of category A take the spectrum of the coefficient matrix (such as A-stability) into the consideration, some other stability definitions take a bit of the solution operator's compactness into account, such as the TVD stability. If we consider the numerical smoothness indicators as a kind of stability indicator, they have the compactness of solution operators taken into account by 100\%. Only in the case of a linear equation solved by a linear scheme, the Lax stability was applied to the distance between two numerical solutions to do error propagation estimation. Hence the sufficiency of the Lax equivalence theorem was proven only for linear equations solved by linear schemes. 

Only stability in category (B) can be directly used in proving the error propagation properties of a scheme, but the error propagation properties of schemes are usually not as good the error propagation properties of differential equations. As for the stability definitions in category (A), they actually help people in obtaining numerical smoothing or smoothness maintenance properties, see \cite{sunRK}. If we review the familiar schemes, many of them have smoothing or smoothness maintenance ability while the designers of them were using numerical stability as a design criterion. 

It should be emphasized that most of the definitions of numerical stability in category (A), although help in making the schemes in smoothing or smoothness maintenance because they require the numerical solutions themselves to be bounded, they do not do enough because those definitions do not directly enforce the boundedness of high order derivatives. Please see \cite{sunRK}. 

A. Stuart and A. Humphries has a review paper \cite{Humphries} on the stability of numerical solutions and schemes. In the article they discussed six kinds of problems: linear decay problems, contractive nonlinear problems, gradient systems, dissipative systems, conservative systems, and Hamiltonian systems. Among these six kinds of problems, only for linear decay and contractive systems, the equations can be hoped to satisfy a one-sided Lipschitz condition and use a scheme with B-stability to obtain a uniform error estimate. The number of schemes having B-stability is very limited.

The vast majority of equations and the desired schemes for each of them individually are not so lucky. Most interesting equations, such as the Van der Pol equation, the Lorenz equation, many semi-linear parabolic equations, conservations laws with shocks, they are not so ideal. Only when the equation and the scheme used are both linear, stability of category (A) is also in category (B), which is why the Lax Equivalence Theorem is limited to linear cases. For example, a scheme can be proven to be dissipative stable, but one cannot use this to prove the scheme is convergent. Hence dissipative stability is merely a necessary condition for convergence. 

There are various kinds of nonlinearities which make the equation non-contractive, but dissipative in long-time scales or a significant part of the domain. Before we have sufficient knowledge about the locations of the original solution and the numerical solution, we cannot utilize these nice facts. If we use the Lipschitz or one-sided Lipschitz condition in a fairly large subdomain with a Lipschitz constant $m>0$, there is going to have excessive exponential error growth in our global error estimate. For example, if $m=5$ and $T=60$, a coefficient $e^{mT} = e^{300} \approx 10^{130}$ will easily submerge $h^p$ for an $h$ and a $p$ which give us a good numerical solution. 

In Chapter 3, we do a case study on the Van der Pol equation, where we show how the numerical smoothness framework overcome the difficulties and get a perfect global error estimate.

\chapter{Smoothness indicators}

In this section we present a simple temporal smoothness indicator for time-dependent ODE and a simple case of spatial smoothness indicator in approximation theory. 

For ODEs or ODE systems, the smoothness of a function in a time domain means two things. (1) The existence of derivatives up to a positive order $k > 1$. (2) All the derivatives up to order $k$ are bounded. Therefore the numerical smoothness for ODE solutions should be given by the smoothness of the $\tilde{u}(t)$ in each step $[t_n,t_{n+1}]$, which is to be established or disqualified in two step. (1) Find whether $\tilde{u}(t)$ is smooth at initial time $t_n$, which is given by an indicator computable from $u_N(t_n)$. (2) Prove $\tilde{u}(t)$ being smooth in the entire step $[t_n,t_{n+1}]$. 

For the spatial indicators, we stay away from differential equations for now. Instead, we will show a necessity result in approximation theory in general.  The result is that numerical smoothness is necessary for a discrete approximation function to converge to {\it any} smooth function. This necessity result has far reaching consequences as it validates a new logical argument.  That is, one can assume the numerical smoothness of a computed solution before doing error estimation, because otherwise the loss of numerical smoothness has already lost the desired convergence for us. On the other side, numerical smoothness should be taken as a criterion when a scheme is designed.  

\section{Temporal smoothness indicator}

For the ODE
$$
u' = f(u),
$$
consider  the first a few successive derivatives
$$ 
v = u', \quad w = v', \quad p = w', \quad q = p'. \label{defDeriv}
$$

Then, using the chain rule, we get
\begin{align*} 
v =u' &=f(u)  \\
w = v' &= \underline{ f'(u) v }\\
p = w' &= f''(u) v^2 + \underline{ f'(u) w }                                             \tag{B} \label{formulaDeriv}  \\
q = p' &= f'''(u) v^3 + 3 f''(u) v w + \underline{  f'(u) p } \\
q' &= f^{(4)}(u) v^4 + 6 f'''(u) v^2 w + 3 f''(u) w^2 + 4 f''(u) v p + \underline{ f'(u) q } \\
  & \cdots \cdots 
\end{align*}
It can be seen that $u, v, w, p, q, \cdots$ are used here to keep the formulas brief and the computations efficient. All these equations can be used as evaluation formulas successively to compute, at each $t_n$,
$$
 \tilde{u}'(t_n^+) = v(t_n^+),  \quad \tilde{u}''(t_n^+)=w(t_n^+), \quad \cdots, \quad \tilde{u}^{(4)}(t_n^+)=q(t_n^+).
$$

If we substituted each derivative by its form in terms of $f(u), f'(u)f(u)$, $f''(u) (f(u))^2 + (f'(u))^2 f(u),  \cdots$, the expressions would be much longer. Moreover, we want to point out that each equation in (B)  is linear in terms of its new unknown function, such as $w' = \cdots +\underline{  f'(u) w } $, where no $w$ is hiding in the dots. This observation tells us that, where $u(t)$ does not blow-up, none of $v, w, p, q, \cdots$ can blow-up. This kind of linearity will be useful in many other ways down the road. 

For any numerical scheme, at $t_n$ with $u_N(t_n)$ obtained,  we are going to define the temporal smoothness indicator in the following way. The fixed order 5 here is used to avoid tedious expressions. In practice, one can choose to go with lower or higher order. 
\begin{definition} \label{defTemporalSI}
For a numerical solution $u_N(t_n)$ of (\ref{DE}), $n=1,2,\cdots, K,$   its temporal smoothness indicator of order k is
$$
T^k = (T_0, T_1, \cdots, T_{K-1}) ,
$$
where
$$
T_n = \left( u_N(t_n),  \tilde{u}'(t_n^+),  \tilde{u}''(t_n^+), \cdots, \tilde{u}^{(k+1)}(t_n^+) \right). 
$$
The $j$-th derivative $\tilde{u}^{(j)}(t_n^+)$  in $T_n$ can be computed by the $j$-th equation in (B), with everything evaluated at $t_n$. 
\end{definition}
{\bf Remark}: In the era of MatLab, Python and AI LLM, the symbolic computation should no longer be considered as something hard to do. 

{\bf Remark}: The definiton also applies to ODE systems. For a system of $m$ unknown functions and $m$ equations, $f(u)$ is a vector of $m$ components. $f'(u)$ is an $m$ by $m$  matrix, which is the Jacobi matrix acting on vector. $f''(u)$ is a cubic matrix, or better called a 3D tenser, acting on 2 vectors consecutively. $f'''(u)$ is a 4D tensor acting on 3 vectors consecutively. $f^{(4)}(u)$ is a 5D tensor, acting on 4 vectors in a row. When $f(u)$ is sufficiently smooth, the orders of the actions can be permuted. 

This indicator gives us the key information on the smoothness of $\tilde{u}(t)$ at the initial time of each step $[t_n,t_{n+1}]$. What's still missing is the smoothness of $\tilde{u}(t)$ in the entire step. We address this issue separately for the cases of non-stiff equations and stiff equations. It is worth to point out that, with the indicator computed, the remaining theoretical analysis is a local analysis in a time step, which should be doable. 

In the case of temporal indicators,  by numerical smoothness we mean that the indicator $(T_0, T_1, \cdots, T_{K-1})$ is bounded. The necessity of the boundedness is obvious, if we want to guarantee that $u_N(t_{n+1})$ to approximate $\tilde{u}(t_{n+1})$ well. There is no need to worry about how bounded is the indicator $T^k$, because the indicator will eventually appear in an error estimate $h^{p+1} F(S^p) + \tau^k G(S^p,T^k)$, or  $\tau^k G(T^k)$ for pure ODEs. 

{\bf A. The cases of non-stiff ODEs}

For non-stiff equations, we start with observing that each equation in (\ref{formulaDeriv})  is in the pattern
$$
\frac{d}{dt}z = A(t) + f'(u) z,
$$
where \(A(t)\) consists of a sum of products of lower-order derivatives of $\tilde{u}(t)$ multiplied by a derivative  of \(f\) (in a low order) evaluated at \(\tilde{u}(t)\) for $t \in [t_n, t_{n+1}]$. Obviously, such an equation is linear in its unknown solution $z$. This observation can be summarize as a trivial theorem, if desired. For the linear equation of $z$, with initial condition 
$$
z(t_n) = \tilde{u}^{(k)}(t_n^+)
$$
and the right choice of $A(t)$, we have $z(t) = \tilde{u}^{(k)}(t)$ in $[t_n, t_{n+1}]$. 
By the method of integrating factors, the solution must be
$$
z(t) =  \tilde{u}^{(k)}(t_n^+) \exp\left(\int_{t_n}^t f'(\tilde{u}(s)) \, ds \right) +  \int_{t_n}^t A(s) \exp\left( \int_{s}^t f'(\tilde{u}(\xi)) \, d\xi \right) \, ds.
$$
When the original DE is assumed as non-stiff, $f'(\tilde{u})$ is well bounded, the exponential part in the equations must be proportional to $\max{|f'(\tilde{u})}|$, then it's easy to see that both $A(t)$ and $z(t)$ must be well bounded. Specifically, the first part of $z(t)$ is close to $\tilde{u}^{(k)}(t_n^+)$, and the second part is close to $0$. Hence the smoothness of $\tilde{u}(t)$ in $[t_n, t_{n+1}]$ is well presented by the smoothness indicator. 

{\bf B. Cases of stiff ODEs}

There are too many different cases of nonlinear stiff ODEs, so we start with considering an example of a linear stiff ODE. We are already familiar with the equation 
$$
u' = \lambda (a(t)-u), \qquad u0)=0.
$$
Earlier, we have studied the equation in the case of $\lambda = 18$ and $a(t) = cos(t)$. In the more general case here, after letting $v=u', w=v', p=w'$ and $q=p'$, we have
\begin{align*}
v' &= \lambda (a'(t)-v), \\
w' &= \lambda (a''(t)-w), \\
p' &= \lambda (a'''(t)-p) ,\\
q' &= \lambda (a^{(4)}(t)-q). 
\end{align*}
If we try to find an upper bound for $\tilde{u}^{(5)}(t) $ by straight forward computation, 
$$
\tilde{u}^{(5)} = q' = a^{(4)} \lambda - a''' \lambda^2 + a'' \lambda^3 -a' \lambda^4 + a \lambda^5 - \lambda^5 \tilde{u}
$$
evaluated for $t \in [t_n,t_{n+1}]$. For a large value of $\lambda$ in the stiff case, with the non-constant $a(t)$, it might be hard to get a good estimate of $\tilde{u}^{(5)}$. However, we haven't taken the advantage of the properties of negative feedback. In fact, we have 
$$
\tilde{u}'(t) = v(t) = \tilde{u}'(t_n^+) e^{-\lambda (t - t_n)} + \lambda \int_{t_n}^t a'(s) e^{-\lambda (t - s)} \, ds ,
$$
$$
\tilde{u}''(t) = w(t) = \tilde{u}''(t_n^+) e^{-\lambda (t - t_n)} + \lambda  \int_{t_n}^t a''(s) e^{-\lambda (t - s)} \, ds ,
$$
$$
\tilde{u}'''(t) = p(t) = \tilde{u}'''(t_n^+) e^{-\lambda (t - t_n)} + \lambda  \int_{t_n}^t a'''(s) e^{-\lambda (t - s)} \, ds ,
$$
$$
\tilde{u}^{(4)}(t) = q(t) = \tilde{u}^{(4)}(t_n^+) e^{-\lambda (t - t_n)} +\lambda  \int_{t_n}^t a^{(4)}(s) e^{-\lambda (t - s)} \, ds .
$$
Similarly, we also have
$$
\tilde{u}^{(5)}(t) = q'(t) = \tilde{u}^{(5)}(t_n^+) e^{-\lambda (t - t_n)} +\lambda  \int_{t_n}^t a^{(5)}(s) e^{-\lambda (t - s)} \, ds .
$$
All these derivatives are to be evaluated in $[t_n,t_{n+1}]$. Since 
$$
\int_{t_n}^t e^{-\lambda (t - s)} \, ds = \frac{1-e^{-\lambda (t-t_n)}}{\lambda}, 
$$
within the step we have, for $ k=1,2,3,4,5$, $|\tilde{u}^{(k)}(t)| $ bounded by the weighted average of $|\tilde{u}^{(k)}(t_n)|$ and $\max|a^{(k)}(t)|$,
$$
|\tilde{u}^{(k)}(t)| \le e^{-\lambda (t-t_n)} |\tilde{u}^{(k)}(t_n)| + (1-e^{-\lambda (t-t_n)}  ) \max_{t \in [t_n,t_{n+1}]} |a^{(k)}(t)| .
$$
For the linear stiff case, the last inequality for each $k$ tells us that the numerical smoothness in the $n$-th step only depends on the smoothness indicator $T_n$.

Now it remains to deal with the general cases of nonlinear stiff equations and systems. Since each concrete problem has its own difficulties, it is probably impossible  to deal with these problems in general. We want to emphasize here that these analyses are local analyses, so one can do things relatively easily.

\section{A spatial smoothness indicator}

As we have said, the spatial smoothness indicator is to be presented in pure approximation theory. Let $\P$ be the space of polynomials of degree $p$ or less, and $\P_h$ be
the space of piecewise polynomials of degree $p$ or less on the uniform
partition $a=x_0 < x_1 < \cdots < x_N=b$, $h=x_{i+1}-x_i=(b-a)/N$.
Let $u^R$ be a piecewise polynomial in $\P_h$.
The reason for using the superscript $R$ is to indicate that $u^R$ might be a reconstructed
numerical solution, possibly from nodal values, cell averages, or other forms of numerical
solutions. Define a smoothness indicator for any $u^R \in \P_h$ as
$$
S^p = (\mathcal{M}, \mathcal{D})
$$
with
$$
\mathcal{M} = ({M}_0, {M}_1, \cdots, {M}_{N-1}),
\qquad \text{where} \quad {M_i} = ( M^0_{i}, M^1_{i}, \cdots,
M^p_{i} ),
$$
and
$$
\mathcal{D} = ({D}_1, \cdots, {D}_{N-1}), \qquad
\text{where} \quad {D}_i = ( D^0_{i}, D^1_{i}, \cdots, D^p_{i}
).
$$
Furthermore,
$$
M_{i}^k = \frac{d^k}{dx^k} u^R(x_i^+), \qquad L_{i}^k =
\frac{d^k}{dx^k} u^R(x_i^-), \qquad J^k_i =  M_{i}^k -  L_{i}^k
$$
and, finally,
$$
D_{i}^k = J_{i}^k / h^{p+1-k}.
$$
{\bf Remark:} This definition of spatial smoothness indicator is taken from \cite{necessity}. In practice one can make some adjustments based on the mesh and the needs of the subsequent error estimations. We choose to show a simplest indicator here to ease the following statement and proof of the theorem on the necessity of numerical smoothness. Since this is a necessity result, it is sufficient to discuss it in the simplest case. Obviously, if it is a necessity in a very simple case, one expects it to be also a necessity in more complex cases. If unsure, one can try to prove a necessity result in a similar way for a complex case, see \cite{chou, chousun}.

In this article, $W^{k}_{q}$ is the standard notation of the Sobolev space of the functions, where
the $k$-th derivative is $L^q$-integrable. $H^k = W^{k}_{2}$. 

\begin{definition}
A piecewise polynomial $u^R \in \P_h$ is numerically
$W^{p+1}_{\infty}$-smooth in the  partition of cell size $h$, if there is a
constant $M_{\infty}$ such that all the components of $S^p$ are
bounded by $M_{\infty}$; $u^R$ is numerically $H^{p+1}$-smooth in
the partition if there is a constant $M_2$ such that
$\sum_{i=1}^{N-1} h \, [(D^0_i)^2 + (D^1_i)^2 + \cdots + (D^p_i)^2]
\le M_2^2$; $u^R$ is numerically $W^{p+1}_{1}$-smooth in the partition if
there is a constant $M_1$ such that $\sum_{i=1}^{N-1} h \, [|D^0_i|
+ |D^1_i| + \cdots + |D^p_i| ] \le M_1$.
\end{definition}
Remark: There is no need to worry about how large the $M_{\infty}$, $M_2$ and $M_1$ are. All components of the indicator $S^p$ will appear in the error estimates. In the practice of numerical computation, no one will make a numerical solution converge. The really important thing is to have an error estimate when the converging sequence is truncated.

\section{Necessity of numerical smoothness}

Portions of this Section expand upon the theoretical framework originally developed in T. Sun, 'Necessity of numerical smoothness,' Int. J. Inf. Sci., 1 (2012), 1–6 (arXiv:1207.3026), \cite{necessity}.

It is obvious that, for the piecewise polynomial $u^R$ to approximate a smooth function
in the interior of all the cells in the domain, one expects
$|M^k_{i}| \le \mathcal{O}(1)$ for all $k$ and $i$. It is also obvious that,
if $u^R$ is an approximation of a smooth function with the optimal
convergence rate, one can expect $|D^0_{i}| \le \mathcal{O}(1)$ for all $i$.
In the case $k=p$, it is obvious that $|D^p_i| \le \mathcal{O}(1)$ implies that
the piecewise constant function $\frac{d^p}{dx^p} u^R(x)$ has
bounded variation. In the main theorem below, we show the necessity
of the boundedness of the other components of the smoothness indicator. Let's start with a simple lemma.

\begin{lemma}
$$
Q(D^0,\cdots,D^p) =  \min_{\hat{v} \in \P} \left(
    \left\|\hat{v}+\frac{1}{2} \sum_{k=0}^{p} \frac{D^k}{k!}\xi^k \right\|^2_{L^2 (-\frac{1}{2},0)}
+ \left\|\hat{v}- \frac{1}{2} \sum_{k=0}^{p} \frac{D^k}{k!}\xi^k \right\|^2_{L^2 (0,\frac{1}{2})}  \right)
$$
is a positive definite quadratic form.
\end{lemma}
\noindent {\bf Proof.}  In the process of calculating the
coefficients of the polynomial
$$
\hat{v} = \sum_{k=0}^{p} V_k \xi^k
$$
to attain the minimum, each $V_k$ must be a linear combination of
$D^0,D^1,\cdots,D^p$, hence the minimum is a quadratic form
of $D^0,D^1,\cdots,D^p$. The quadratic form is positive
definite because, if any entry $D^k$ of $(D^0,D^1,\cdots,D^p)$ is non-zero, then we have either
$V_k+\frac{1}{2}\frac{D^k}{k!} \ne 0$ or $V_k-\frac{1}{2}\frac{D^k}{k!} \ne 0$.
Due to the linear independence of $\{1,\xi,\cdots,\xi^p\}$, one of
the two terms in the minimum must be positive.  \#

\begin{theorem} \label{T2}
Suppose that $u \in H^{p+1}(a,b)$, and $u^R$ is a
piecewise polynomial function described as above. Then, there is a
constant $C_2 > 0$,  independent of $h$, $u$ and $u^R$, such that
\begin{equation} \label{L2D}
\|u - u^R\|_{L^2(a,b)} \ge  h^{p+1} \left[ \sqrt{\sum_{i=1}^{N-1} h
\, Q(D^0_i,D^1_i,\cdots,D^p_i)} - C_2 |u|_{H^{p+1}(a,b)} \right].
\end{equation}
\end{theorem}

\noindent {\bf Proof.} Let $\P_c$ be the space of piecewise
polynomials of degree $p$ or less on the partition
$a<x_{\frac{1}{2}}<x_{\frac{3}{2}}<\cdots<x_{N-\frac{1}{2}}<b$,
where $x_{i-\frac{1}{2}} = (x_{i-1}+x_i)/2$. Since $u \in
H^{p+1}(a,b)$, there is a constant $C_2 > 0$, independent of $h$ and
$u$, and a piecewise polynomial $u^I \in \P_c$, such that
\begin{equation} \label{uI}
\|u-u^I\|_{L^2 (a,b)} \le C_2 h^{p+1} |u|_{H^{p+1}(a,b)}.
\end{equation}
Now, since $\|u^I-u^R\|_{L^2 (a,b)} \le \|u^I-u\|_{L^2 (a,b)} + \|u-u^R\|_{L^2 (a,b)}$, we have
\begin{eqnarray}
&& \|u-u^R\|_{L^2 (a,b)} \ge \|u^I-u^R\|_{L^2 (a,b)} - \|u-u^I\|_{L^2 (a,b)} \nonumber \\
&& \ge \sqrt{ \sum_{i=1}^{N-1}  \|u^I-u^R\|^2_{L^2
(x_{i-\frac{1}{2}},x_{i+\frac{1}{2}})} } - C_2 h^{p+1}
|u|_{H^{p+1}(a,b)}. \label{jp}
\end{eqnarray}
Let $Q^k_i = (M^k_i + L^k_i)/2$ and $w(x) = \sum_{k=0}^p
\frac{Q^k_i}{k!} (x-x_i)^k$. Then,
$$
u^R-w = \frac{1}{2} \sum_{k=0}^p \frac{J^k_i}{k!} (x-x_i)^k
$$
for $x \in \Delta_i^+ = (x_i,x_{i+\frac{1}{2}})$, and
$$
u^R-w = - \frac{1}{2} \sum_{k=0}^p \frac{J^k_i}{k!} (x-x_i)^k
$$
for $x \in \Delta_i^- = (x_{i-\frac{1}{2}},x_i)$.
\begin{eqnarray}
&& \|u^I-u^R\|^2_{L^2 (x_{i-\frac{1}{2}},x_{i+\frac{1}{2}})} \ge \min_{v \in \P}  \|v-u^R\|^2_{L^2 (x_{i-\frac{1}{2}},x_{i+\frac{1}{2}})} \nonumber \\
&& =  \min_{v \in \P}  \|v-(u^R-w)\|^2_{L^2 (x_{i-\frac{1}{2}},x_{i+\frac{1}{2}})} \nonumber \\
&& =  \min_{v \in \P} \left( \left\|v+\frac{1}{2} \sum_{k=0}^{p}
\frac{J^k_i}{k!}(x-x_i)^k \right\|^2_{L^2(\Delta_i^-)}
                                 + \left\|v-\frac{1}{2} \sum_{k=0}^{p} \frac{J^k_i}{k!}(x-x_i)^k \right\|^2_{L^2 (\Delta_i^+)}  \right) \nonumber \\
&& =  h^{2p+3}  \min_{\hat{v} \in \P} \left(
\left\|\hat{v}+\frac{1}{2} \sum_{k=0}^{p} \frac{D^k_i}{k!}\xi^k
\right\|^2_{L^2 (-\frac{1}{2},0)}
              + \left\|\hat{v}-\frac{1}{2} \sum_{k=0}^{p} \frac{D^k_i}{k!}\xi^k \right\|^2_{L^2 (0,\frac{1}{2})}  \right) \nonumber \\
&& = h^{2p+3} \, Q(D^0_i, D^1_i,\cdots,D^p_i). \label{jpL}
\end{eqnarray}
Plugging (\ref{jpL}) into (\ref{jp}), we have proven (\ref{L2D}). \#

\begin{theorem} \label{T1}
If $u \in W^{p+1}_{1}(a,b)$, and $u^R$ is as described previously,
then, there are positive constants $C_1$ and $C_{12}^p$, independent
of $h$, $u$ and $u^R$, such that
\begin{equation} \label{L1D}
\|u - u^R\|_{L^1(a,b)} \ge  h^{p+1} \left[ C_{12}^p \sum_{1\le i \le
N-1} h \,\sqrt{ Q(D^0_i,D^1_i,\cdots,D^p_i)} - C_1
|u|_{W^{p+1}_{1}(a,b)} \right].
\end{equation}
\end{theorem}
\noindent {\bf Proof.} Since $u \in W^{p+1,1}(a,b)$, there is a
piecewise polynomial $u^I \in \P_c$, such that
$$
\|u-u^I\|_{L^1 (a,b)} \le  C_1 h^{p+1} |u|_{W^{p+1}_{1}(a,b)},
$$
where $C_1 >0$ is a constant independent of $h$ and $u$. Obviously,
\begin{eqnarray}
&& \|u-u^R\|_{L^1 (a,b)} \ge \|u^I-u^R\|_{L^1 (a,b)} - \|u-u^I\|_{L^1 (a,b)} \nonumber \\
&& \ge  \sum_{1 \le i \le N-1} \|u^I-u^R\|_{L^1
(x_{i-\frac{1}{2}},x_{i+\frac{1}{2}})}  - C_1 h^{p+1}
|u|_{W^{p+1}_{1}(a,b).} \label{jp1}
\end{eqnarray}
By the standard scaling argument, and (\ref{jpL}) which remains valid for the current $u^I$, it is easy to
prove that there is a constant $C_{12}^p$ such that
$$
\|u^I-u^R\|_{L^1 (x_{i-\frac{1}{2}},x_{i+\frac{1}{2}})} \ge C_{12}^p
h^{\frac{1}{2}} \|u^I-u^R\|_{L^2
(x_{i-\frac{1}{2}},x_{i+\frac{1}{2}})} \ge C_{12}^p h^{p+2}
\sqrt{Q(D^0_i,\cdots,D^p_i)}.
$$
Plugging this into (\ref{jp1}), we have (\ref{L1D}) proven. \#


\begin{theorem} \label{Tinfinity}
If $u \in W^{p+1}_{\infty}[a,b]$, and $u^R$ is as described previously, then,
there is a constant $C_\infty>0$, independent of $h$, $u$ and $u^R$,
such that
\begin{equation} \label{L8D}
\|u - u^R\|_{L^{\infty}(a,b)} \ge  h^{p+1} \left[\max_{1\le i \le
N-1} \sqrt{ Q(D^0_i,D^1_i,\cdots,D^p_i)} - C_\infty
|u|_{W^{p+1}_{\infty}[a,b]} \right].
\end{equation}
\end{theorem}
\noindent {\bf Proof.}
Since $u \in W^{p+1}_{\infty}[a,b]$, there is a piecewise polynomial $u^I \in \P_c$, such that
$$
\|u-u^I\|_{L^{\infty} (a,b)} \le  C_\infty h^{p+1}
|u|_{W^{p+1}_{\infty}[a,b]},
$$
where $C_\infty >0$ is a constant independent of $h$ and $u$.
Obviously,
\begin{eqnarray}
&& \|u-u^R\|_{L^{\infty} (a,b)} \ge \|u^I-u^R\|_{L^{\infty} (a,b)} - \|u-u^I\|_{L^{\infty} (a,b)} \nonumber \\
&& \ge  \max_{1 \le i \le N-1}
\|u^I-u^R\|_{L^{\infty}(x_{i-\frac{1}{2}},x_{i+\frac{1}{2}})}  -
C_\infty h^{p+1} |u|_{W^{p+1}_{\infty}[a,b].} \label{jp8}
\end{eqnarray}
Now, let $U_i =  \|u^I-u^R\|_{L^{\infty}(x_{i-\frac{1}{2}},x_{i+\frac{1}{2}})}$, then, by using (\ref{jpL}) for the current $u^I$,
\begin{eqnarray}
&& U_i = \sqrt{ \int_{x_{i-\frac{1}{2}}}^{x_{i+\frac{1}{2}}} \frac{1}{h} U_i^2 dx } \ge  \sqrt{ \int_{x_{i-\frac{1}{2}}}^{x_{i+\frac{1}{2}}} \frac{1}{h} (u^I-u^R)^2 dx } \nonumber \\
&& =  \sqrt{ \frac{1}{h} \|u^I-u^R\|^2_{L^2 (x_{i-\frac{1}{2}},x_{i+\frac{1}{2}})}  }
\ge h^{p+1} \sqrt{ Q(D^0_i,D^1_i,\cdots,D^p_i)}. \label{D8}
\end{eqnarray}
By plugging (\ref{D8}) into (\ref{jp8}), we have proven (\ref{L8D}).
\#

{\bf A few conclusion remarks:}

According to Theorem \ref{T2}, in order to have
$$
\|u-u^R\|_{L^2(a,b)} \le \mathcal{O}(h^{p+1})
$$
for any function $u \in H^{p+1}(a,b)$, $u^R$ must be numerically $H^{p+1}$-smooth. That is
\begin{equation} \label{Hsmoothness}
\sum_{i=1}^{N=1} h\,  Q(D^0_i,D^1_i,\cdots,D^p_i) \le \mathcal{O}(1).
\end{equation}
According to Theorem \ref{T1}, in order to have
$$
\|u-u^R\|_{L^1(a,b)} \le \mathcal{O} (h^{p+1})
$$
for any function $u \in W^{p+1}_{1}(a,b)$, $u^R$ must be numerically
$W^{p+1}_{1}$-smooth. That is
\begin{equation} \label{HHsmoothness}
\sum_{i=1}^{N=1} h\, \sqrt{ Q(D^0_i,D^1_i,\cdots,D^p_i) } \le \mathcal{O}(1).
\end{equation}
According to Theorem \ref{Tinfinity}, in order to have
$$
\|u-u^R\|_{L^{\infty}(a,b)} \le \mathcal{O}(h^{p+1})
$$
for any function $u \in W^{p+1}_{\infty}[a,b]$, $u^R$ must be numerically $W^{p+1}_{\infty}$-smooth. That is
\begin{equation} \label{Csmoothness}
Q(D^0_i,D^1_i,\cdots,D^p_i) \le \mathcal{O}(1)
\end{equation}
for all $i$. Because $Q(d_0,d_1,\cdots,d_p)$ is a positive definite quadratic form, the last inequality is equivalent to that $|D^k_i| \le \mathcal{O}(1)$ for all $k$ and $i$. This is what we mean by the necessity of the numerical smoothness. To be more explicit, $|D^k_i| \le \mathcal{O}(1)$ is equivalent to $|J^k_i| \le \mathcal{O}(h^{p+1-k})$. That is, the jumps of the $k$-th derivative of $u^R$ need to be as small as $\mathcal{O}(h^{p+1-k})$, for $k=0,1,\cdots,p$.

Although the results and the proofs of the theorems seem to be very
simple, the author believes that the impact of the theorems could be
far reaching. For a numerical solution of a time-dependent PDE, the
theorem implies that a fully-discrete scheme must enforce numerical
smoothness, otherwise the convergence of optimal rate must have been
lost. Consequently, some of the traditional numerical stability
notions are possibly inadequate, if they do not enforce numerical
smoothness. 

For the two dimensional case, S. H. Chou proved similar results, \cite{chou, chousun}. For one dimensional finite volume methods (WENO), T. Sun defined spatial smoothness indicators based on piecewise polynomial (degree 5) reconstruction from the cell-averages. The reconstruction is just like discontinuous finite elements. A lot of work can be done to cover different cases of approximation methods, but the necessity of numerical smoothness should be a valid criterion for an approximation to be close to a smooth function.

\section{Simultaneous approximation}












With the method of proving the necessity theorem of the last section, we can also prove a simultaneous approximation theorem. The well-known Bramble-Hilbert Lemma proved the existence of a simultaneous approximation of order $p$ for a smooth function having bounded $p+1$ derivative. In the theorem below it is proven that, if a polynomial $p(x)$ of degree $p$ approximates a smooth function $f(x)$ at its optimal rate, then the simultaneous approximation must be valid for all the $k$-th derivatives of $p(x)$ for $k= 1,2,\cdots, p$.  

\begin{theorem} 
 Let $f(x) \in C^{p+1}[0,h]$ for a given positive integer $p$. Suppose there is  a polynomial of degree $p$ (or less) and  a constant $C_0$, such that
$$
|f(x) - p(x)| \leq C_0 h^{p+1}
$$
for all $x \in [0,h]$. Then there must be constants $C_1, C_2, \cdots, C_p$, such that
$$
|f^{(k)}(x) - p^{(k)}(x)| \leq C_k  h^{p+1-k}
$$
for all $x \in [0,h]$ and $k=1, 2, \cdots, p$.
\end{theorem}

\noindent {\bf Proof.} For an arbitrary point $\xi \in [0,h]$, let $T(x)$ be the  Taylor expansion of degree $p$ of $f(x)$ at $\xi$,
$$
T(x) = f(\xi) + f'(\xi) (x-\xi) + \cdots + \frac{f^{(p)}(\xi)}{p!} (x-\xi)^p ,
$$
then
$$
|f(x) - T(x)| \leq M_{p+1} h^{p+1}
$$
for all $x \in [0,h]$ , where $M_{p+1} = \max_{x \in [0,h]} \frac{|f^{(p+1)}(x)|}{(p+1)!}$. Hence
$$
|P(x) - T(x)| \leq \left(C_0+M_{p+1} \right) h^{p+1}
$$
for all $x \in [0,h]$. Let
$$
w(x) = p(x) - T(x).
$$
Obviously, for $k=0,1,\cdots,p$, 
$$
w^{(k)}(\xi) = p^{(k)}(\xi) - f^{(k)}(\xi)
$$
and, for $C_w = (C_0 + M_{p+1} )^2$ and
$$
\int_0^h w^2(x) \, dx \le C_w  h^{2p+3}.
$$
Consider a change of variable given by 
$$
x = hy, \qquad \eta = h \xi, \qquad z(y) = w(hy) = w(x).
$$
Consequently,
$$
z'(y) = \frac{d}{dy} w(hy) = h w'(hy) = h w'(x), 
$$
and
$$
z^{(k)}(y) = h^k w^{(k)}(x), \qquad k=2,\cdots, p.
$$
Moreover, 
$$
\int_0^1 z^2(y) \, dy = \frac{1}{h} \int_0^h w^2(x) \, dx \le C_w h^{2p+2}.
$$
Since $w(x) = p(x)-T(x)$ is a polynomial of degree $p$, so is $z(y)$, hence 
$$  
z(y) = z(\eta) + z'(\eta) (y-\eta) + \cdots + \frac{z^{(p)}(\eta)}{p!} (y-\eta)^p .    
$$  

Consider 
$$
Q_{\eta}(a_0, a_1, \cdots, a_p) = \int_0^1 \left(a_0 + a_1 (y-\eta) + \cdots + \frac{a_p}{p!} (y-\eta)^p \right)^2 dy.
$$
This is obviously a symmetric positive definite quadratic form for each $\eta$ in $[0,1]$, since $\{ 1, y-\eta, \cdots, (y-\eta)^p \} $ is linearly independent. Moreover, since $Q_{\eta}$ depends on $\eta$ continuously, it only takes an elementary proof to show there is a positive constant $C_q$ such that 
$$
|a_0|^2 + |a_1|^2 + \cdots + |a_p|^2   \leq C_q \, Q_{\eta} (a_0, a_1,\cdots, a_p)
$$
uniformly for all $\eta$ in $[0,1]$. Consequently, replacing $a_j$ by $z^{(j)}(\eta)$ for $j = 0, 1, \cdots, p$,  we have
\begin{eqnarray*}
|z(\eta)|^2 + |z'(\eta)|^2 + \cdots + |z^{(p)}(\eta)|^2   & \leq C_q \, Q_{\eta}(z(\eta), z'(\eta),\cdots,z^{(p)}(\eta))  \\
                                                                                      & = C_q \int_0^1 z^2(y) dy \leq C_q C_w  h^{2p+2}.
\end{eqnarray*}
Finally we get 
$$
| p^{(k)}(\xi) - f^{(k)}(\xi) | = | w^{(k)}(\xi) | =  h^{-k} \, | z^{(k)}(\eta) | \leq \sqrt{C_q C_w}  h^{p+1-k}. \#
$$

A simple application of this theorem can be that, once we can guarantee the optimal approximation of a solution of a differential equation, we can be confident that we have basically eliminated the possibility of any excessive oscillations in our solution. 

The theorem only covers the 1-D case, while the proof is elementary and limited to the continuous function space $C^{p+1}_{\infty}$. For the general d-dimensional Sobolev spaces, one can use the Bramble-Hilbert Lemma and inverse inequalities to prove a general theorem in $W^{p+1}_q(\Omega)$.

\chapter{Case study: Van der Pol system}

In this chapter we present an example of solving a nonlinear ODE system, namely the system of the well known Van der Pol equation (VdP). We are going to solve it using the classical Runge-Kutta method of order 3 (RK3), compute the smoothness indicator, adaptively control the local error, and estimate the global error for long term. This will show how the nonlinearity of the Van der Pol system can be dealt with in the numerical smoothness framework. The reason for choosing RK3 is for its relatively high order and its moderately tedious local error derivation among the Runge-Kutta schemes. After all, we want to focus on the error analysis framework. 

The Van der Pol equation (VdP) is a second order nonlinear equation
\begin{equation} \label{VDP}
x'' - \mu (1-x^2) x' + x = 0, \qquad x(0)=x_0, \quad x'(0)=y_0.
\end{equation}
Here $\mu>0$ is a parameter controlling the damping and, mathematically, changing the stiffness level of the equation. 

Another important issue with the VdP is its self energization when the damping coefficient $-\mu (1-x^2)$ is negative, or $x^2<1$. While large parameter $\mu$ makes the system highly stiff, $x^2<1$ makes the error of a numerical solution growing fast exponentially along with the solution itself, although the fast growth is followed immediately by fast exponential reduction when $-\mu (1-x^2) > 0$. Both the strong stiffness and the fast exponential growth/decay put the error propagation analysis based on Lipschitz conditions in a nightmare. 

As usual, we convert the second order equation to a first order system by introducing $y=x'$. Now, we have the system 
\begin{align}
x' &= y,  \\
y' &= \mu (1-x^2) y - x,
\end{align}
with $x(0)=x_0$ and $y(0)=y_0$.

For the briefness of notations, one can use vectors 
$$
\mathbf{u} = \begin{bmatrix} x \\ y \end{bmatrix} , \qquad \mathbf{u}' = \begin{bmatrix} x' \\ y' \end{bmatrix} , \qquad\mathbf{f}(\mathbf{u}) = \begin{bmatrix} y \\ \mu(1-x^2)y-x \end{bmatrix}, \qquad \mathbf{u}_0 = \begin{bmatrix} x_0 \\y_0 \end{bmatrix}.
$$
In the vector notation, the equation becomes
$$
  \mathbf{u}' =  \mathbf{f}(\mathbf{u}), \qquad \mathbf{u}(0) = \mathbf{u}_0.
$$

We are going to solve the system with the classical Runge-Kutta scheme RK3. For an autonomous system, the scheme is 
\begin{align}
\mathbf{k}_1 &= \mathbf{f}(\mathbf{u}_n)    \nonumber \\
\mathbf{k}_2 &= \mathbf{f} (\mathbf{u}_n + 0.5 \tau \, \mathbf{k}_1 )    \nonumber \\
\mathbf{k}_3 &= \mathbf{f} (\mathbf{u}_n - \tau \,\mathbf{k}_1 + 2 \tau \, \mathbf{k}_2 )  \nonumber \\
\mathbf{u}_{n+1} &= \mathbf{u}_n + \frac{\tau}{6} \left( \mathbf{k}_1 + 4 \mathbf{k}_2 + \mathbf{k}_3 \right)   \label{RK3}
\end{align}
at $t_n$ with step size $\tau=t_{n+1}-t_n$, and $\mathbf{u}_n = \mathbf{u}_N(t_n)$ is the numerical solution. When we use adaptive time stepping, $\tau$ is allowed to change from step to step. 

We start from showing the graphics of a carefully computed solutions to point out the issues we face in the error analysis. For the eventual purpose doing the global error analysis, we will show \\
(a) How to compute the needed high order derivatives of $\mathbf{u}$ and $\mathbf{f}(\mathbf{u})$; \\
(b) How to do the local error estimation; \\ 
(c) How to deal with the nonlinearity and utilize the VdP's stability property; \\
(d) How to estimate error propagation in a time step; \\
(e) How to estimate the global error for a short time; and \\
(f) How to estimate long term global error near the limit cycle.

At this point, let's focus on Figure \ref{figVdP}. The upper right plot is a numerical solution of the VdP with $\mu = 3$, for $t \in [0, 23]$, in the phase plan (x, y). Its initial point is above the limit cycle. Its end point is on the limit cycle in the plot. After merging into the limit cycle, the solution traveled nearby the limit cycle for about 2 periods. We consider the properties of the solution orbit and the limit cycle, but only need to discuss what happens while $y>0$, because of the symmetry of the solutions in the upper half and the lower half of the phase plane. 

Since $x' = y$ and $y' = \mu (1-x^2) y -x$, we need to consider three regions of the upper half of the phase plane  ($y>0$) :

\begin{itemize}
\item For $x<-1$: Because $\mu (1-x^2)$ is negative here, $y(t)$ shall have exponential decay toward the almost horizontal curve $\mu (1-x^2) y - x = 0$, especially when $\mu$ is large, the contribution of the term $-x$ is almost negligible. Meanwhile, $x' = y$ should be or become close to $ 0$, leaving the solution creeping slowly to the right along the attracting part of the limit cycle. The major challenge in this region is the stiffness when we have a large $\mu$, but the extremely small $(x',y')$ along the limit cycle may help us reducing the harm of the stiffness. 
\item For $-1<x<1$: The major trouble here is the positive $\mu (1-x^2) y$ in the equation about $y$. This causes positive exponential growth of $y$, which in turn makes $x$ also growing exponentially. Although the plot of the limit cycle in the phase plane looks very smooth in this region, the fast moving $(x,y)$ and the fast growing $(x', y')$ make both the local error and the error propagation rate becoming larger exponentially. Fortunately, the fast motion of the solution makes the solution out of this region in a relatively short time.
\item For $1<x$: Now $\mu (1-x^2) y$ becomes negative again while $y$ starts almost at its maximum. The motion of $(x,y)$ slows down as an exponential decay, but the motion speed becomes close to $0$ when $y$ is near $0$. The concern of numerical overshooting at the sharp turn should not be as bad as the sharp turn suggesting to us. 
\end{itemize}
Overall, the biggest problem we are facing here is the exponential growth/decay and the damage done to the traditional error propagation analysis. The biggest help is that the system has lasting deep attraction to the part of the limit cycle where $x<-1, y>0$.

\section{Computing high order derivatives }

For a system $\mathbf{u}' = \mathbf{f}(\mathbf{u})$ with $\mathbf{u} \in \mathbb{R}^n$ and $\mathbf{f}: \mathbb{R}^n \to \mathbb{R}^n$, define
$$
\mathbf{v} = \mathbf{u}',\quad \mathbf{w} = \mathbf{v}' = \mathbf{u}'',\quad \mathbf{p} = \mathbf{w}' = \mathbf{u}''',
$$
$$
 \mathbf{q} = \mathbf{p}' = \mathbf{u}^{(4)}, \quad \mathbf{r} = \mathbf{q}' = \mathbf{u}^{(5)}.
$$

Below are the first through fifth derivatives of $\mathbf{u}$ with respect to $t$ for the system $\mathbf{u}' = \mathbf{f}(\mathbf{u})$, written in tensor form.
Here $\mathbf{u}^{(k)}$ denotes the k-th derivative of $\mathbf{u}$ with respect to $t$, and $\mathbf{f}^{(k)}$ is the k-th derivative tensor of $\mathbf{f}$ with respect to $\mathbf{u}$. 
\begin{align*}
\mathbf{u}^{(1)} &= \mathbf{v} = \mathbf{f}(\mathbf{u}), \\[4pt]
\mathbf{u}^{(2)} &= \mathbf{w} = \mathbf{v}' = \mathbf{f}'(\mathbf{u})\,\mathbf{v}, \\[4pt]
\mathbf{u}^{(3)} &= \mathbf{p} = \mathbf{w}' = \mathbf{f}''(\mathbf{u})(\mathbf{v}, \mathbf{v}) + \mathbf{f}'(\mathbf{u})\,\mathbf{w}, \\[4pt]
\mathbf{u}^{(4)} &= \mathbf{q} = \mathbf{p}' = \mathbf{f}'''(\mathbf{u})(\mathbf{v}, \mathbf{v}, \mathbf{v}) + 3\,\mathbf{f}''(\mathbf{u})(\mathbf{v}, \mathbf{w}) + \mathbf{f}'(\mathbf{u})\,\mathbf{p}, 
\end{align*}
\begin{align*}
\mathbf{u}^{(5)}  = \mathbf{r} =\mathbf{q}' = &= \mathbf{f}^{(4)}(\mathbf{u})(\mathbf{v}, \mathbf{v} , \mathbf{v}, \mathbf{v}) \\
&+ 6\,\mathbf{f}'''(\mathbf{u})(\mathbf{v}, \mathbf{v}, \mathbf{w}) \\
&+ 3\,\mathbf{f}''(\mathbf{u})(\mathbf{w}, \mathbf{w}) \\
&+ 4\,\mathbf{f}''(\mathbf{u})(\mathbf{v}, \mathbf{p}) \\
&+ \mathbf{f}'(\mathbf{u})\,\mathbf{q},
\end{align*}
and, for those readers wanting to see the trend, 
\begin{align*} 
\mathbf{u}^{(6)} = \mathbf{r}' &= \;\mathbf{f}^{(5)}(\mathbf{u})\bigl(\mathbf{v}, \mathbf{v}, \mathbf{v}, \mathbf{v} , \mathbf{v} \bigr) \\
&+ 10\;\mathbf{f}^{(4)}(\mathbf{u})\bigl(\mathbf{v}, \mathbf{v}, \mathbf{v}, \mathbf{w}\bigr) \\
&+ 15\;\mathbf{f}'''(\mathbf{u})\bigl(\mathbf{v}, \mathbf{w}, \mathbf{w}\bigr) \\
&+ 10\;\mathbf{f}'''(\mathbf{u})\bigl(\mathbf{v}, \mathbf{v}, \mathbf{p}\bigr) \\
&+ 10\;\mathbf{f}''(\mathbf{u})\bigl(\mathbf{w}, \mathbf{p}\bigr) \\
&+ 5\;\mathbf{f}''(\mathbf{u})\bigl(\mathbf{v}, \mathbf{q}\bigr) \\
&+ \mathbf{f}'(\mathbf{u})\,\mathbf{r},
\end{align*}

where $\bf{f}^{(k)}$ denotes the k-th derivative tensor (a symmetric (k,1)-tensor).   In explicit component form, for indices $i,j,\cdots,k$, the contractions are interpreted as
$$
\mathbf{f}^{(k)}(\mathbf{u})(\mathbf{a}_1,\dots,\mathbf{a}_k)_i = \sum_{j_1,\dots,j_k} \frac{\partial^k f_i}{\partial u_{j_1}\cdots\partial u_{j_k}}(\mathbf{u})\, a_{1,j_1}\cdots a_{k,j_k}.
$$
To get familiar with tensors and their operations, one can learn from AI-LLM, or find a textbook. 

It is worth to notice that $\mathbf{f}'(\mathbf{u})$ is the well-known Jacobian matrix of $\mathbf{f}(\mathbf{u})$, which has its usefulness in many ways.

Assume $\mathbf{f}$ to be smooth. Obviously, the boundedness of all the derivatives of $\mathbf{f}(\mathbf{u})$ follows the boundedness of $\mathbf{u}$.  

It is useful for the local error derivation of RK3 to expand $\mathbf{u}^{(3)}$ and $\mathbf{u}^{(4)}$ further by substituting 
$$
\mathbf{v} = \mathbf{f}(\mathbf{u}) = \mathbf{f} , \qquad \mathbf{w}=\mathbf{f}'(\mathbf{u}) \mathbf{f}(\mathbf{u}) = \mathbf{f}' \, \mathbf{f}  
$$
to get
$$
\mathbf{u}^{(3)} = \mathbf{p} = \mathbf{f}'' \, (\mathbf{f},\mathbf{f}) + \mathbf{f}' \mathbf{f}' \mathbf{f}
$$
and
$$
\mathbf{u}^{(4)} = \mathbf{f}''' \, (\mathbf{f}, \mathbf{f}, \mathbf{f}) + 3 \mathbf{f}'' \, (\mathbf{f}, \mathbf{f}' \,\mathbf{f} ) + \mathbf{f}' \mathbf{f}'' (\mathbf{f}, \mathbf{f}) + \mathbf{f}' \mathbf{f}' \mathbf{f}' \mathbf{f}.  
$$
In the derivation of the the local error of RK3, we will need the expansion terms in $\mathbf{u}^{(4)}$. For higher order Runge-Kutta schemes, the number of expansion terms will increase extremely fast. In the case of RK4, we will need the expansion terms of $\mathbf{u}^{(5)}$, but there are 9 of them. See \cite{wanner}.

Another usefulness of the formulas is that we can prove the boundedness of the high order derivatives. It is worth to mention that each $\mathbf{u}^{(k)}$ is the solution of a first order linear ODE. This observation is very important for the validation of the boundedness of the high order derivatives in the time step $[t_n, t_{n+1}]$ from the computed smoothness indicators at $t_n$.

\section{Local error analysis}

To derive the local truncation error of the classical third-order Runge-Kutta (RK3) scheme for an autonomous system, we expand both the true solution $\mathbf{\tilde{u}}(t_{n+1})$ and the numerical solution $\mathbf{u}_N(t_{n+1})$ at $t_n$ using Taylor series up to $\mathcal{O}(h^4)$, then find the difference. 

\vspace{0.1cm}

\noindent {\bf 1. Taylor Expansion of the $\mathbf{\tilde{u}}(t)$ }

For the true solution $\mathbf{\tilde{u}}(t)$ at $t_n$, by the Taylor series, with $\tau=t_{n+1}-t_n$:
$$
\mathbf{\tilde{u}}(t_{n+1}) = \mathbf{\tilde{u}}(t_n) + \tau \mathbf{\tilde{u}'}(t_n) + \frac{\tau ^2}{2} \mathbf{\tilde{u}''}(t_n) +\frac{\tau ^3}{6} \mathbf{\tilde{u}^{(3)}}(t_n)
                            + \frac{\tau ^4}{24} \mathbf{\tilde{u}^{(4)}}(t_n) + \mathcal{O} (\tau ^5). 
$$

For the autonomous system $\mathbf{u}' = \mathbf{f}(\mathbf{u})$, we compute successive total derivatives with respect to time $t$ using the chain rule. All functions and their derivatives are evaluated at $\mathbf{u}_n = \mathbf{u}_N(t_n)$:
\begin{itemize}
    \item \textbf{First derivative:} $\mathbf{\tilde{u}}' = \mathbf{f}$
    \item \textbf{Second derivative:} $\mathbf{\tilde{u}}'' = \mathbf{f}'\mathbf{f}$
    \item \textbf{Third derivative:} $\mathbf{\tilde{u}}''' = \mathbf{f}''(\mathbf{f}, \mathbf{f}) + \mathbf{f}' \mathbf{f}' \mathbf{f}$
    \item \textbf{Fourth derivative:} $\mathbf{\tilde{u}}^{(4)} = \mathbf{f}'''(\mathbf{f},\mathbf{f},\mathbf{f}) + 3\mathbf{f}''(\mathbf{f}'\mathbf{f}, \mathbf{f}) + \mathbf{f}'\mathbf{f}''(\mathbf{f},\mathbf{f}) + \mathbf{f}' \mathbf{f'} \mathbf{f}' \mathbf{f}$
\end{itemize}

Substituting these into the Taylor expansion for the exact solution $\mathbf{\tilde{u}}(t_{n+1})$:
\begin{align*}
\mathbf{\tilde{u}}(t_{n+1}) = \mathbf{\tilde{u}}(t_n) &+ \tau \mathbf{f} + \frac{\tau^2}{2} \mathbf{f}' \mathbf{f} + \frac{\tau^3}{6} [ \,\mathbf{f}''[\mathbf{f},\mathbf{f}) + \mathbf{f}' \mathbf{f}' \mathbf{f} \, ] \\
            &+ \frac{\tau^4}{24} [ \, \mathbf{f}''' \, (\mathbf{f}, \mathbf{f}, \mathbf{f}) + 3 \mathbf{f}'' \, (\mathbf{f}, \mathbf{f}' \,\mathbf{f} ) + \mathbf{f}' \mathbf{f}'' (\mathbf{f}, \mathbf{f}) + \mathbf{f}' \mathbf{f}' \mathbf{f}' \mathbf{f} \, ] + \mathcal{O}(\tau^5).
\end{align*}

\noindent {\bf 2. Taylor Expansion of the RK3 Numerical Solution}

The stages for the RK3 scheme are defined as:
\begin{align*}
    \mathbf{k}_1 &= \mathbf{f}(\mathbf{u}_n) \\
    \mathbf{k}_2 &= \mathbf{f}\left(\mathbf{u}_n + 0.5 \, \tau \mathbf{k}_1\right) \\
   \mathbf{k}_3 &= \mathbf{f}(\mathbf{u}_n - \tau \mathbf{k}_1 + 2\tau \mathbf{k}_2) \\
    \mathbf{u}_{n+1} &= \mathbf{u}_n + \tau \, \left(\frac{1}{6} \, \mathbf{k}_1 + \frac{2}{3} \, \mathbf{k}_2 + \frac{1}{6} \, \mathbf{k}_3 \right)
\end{align*}

\noindent {\bf Expanding $\mathbf{k}_2$}:  By Taylor expansion at $u_n$, 
\[
\mathbf{k}_2 = \mathbf{f}\left(\mathbf{u}_n + \frac{\tau}{2} \mathbf{f}\right) = \mathbf{f} + \frac{\tau}{2} \mathbf{f}'\mathbf{f} + \frac{\tau^2}{8} \mathbf{f}''(\mathbf{f},\mathbf{f}) + \frac{\tau^3}{48} \mathbf{f}'''(\mathbf{f},\mathbf{f},\mathbf{f}) + \mathcal{O}(\tau^4)
\]

\noindent {\bf Expanding $\mathbf{k}_3$}: The argument inside $\mathbf{k}_3$ is
\begin{align*}
\mathbf{u}_n + \tau (-\mathbf{k}_1 + 2 \mathbf{k}_2) &= \mathbf{u}_n + \tau \left(-\mathbf{f} + 2\left[\mathbf{f} + \frac{1}{2}\tau \mathbf{f}'\mathbf{f} + \frac{1}{8} \tau^2\mathbf{f}''(\mathbf{f},\mathbf{f}) + \mathcal{O}(\tau^3)\right]\right) \\
                                  &= \mathbf{u}_n + \tau \mathbf{f} + \tau^2\mathbf{\mathbf{f}}'\mathbf{f} + \frac{1}{4} \tau^3 \mathbf{f}''(\mathbf{f},\mathbf{f}) + \mathcal{O}(\tau^4)
\end{align*}
Letting $\Delta = \tau \mathbf{f} + \tau^2 \mathbf{f}' \mathbf{f} + \frac{1}{4} \tau^3 \mathbf{f}''(\mathbf{f}, \mathbf{f})$, we expand $\mathbf{k}_3 = \mathbf{f}( \mathbf{u}_n + \Delta)$ via multi-variable Taylor series:
\[
\mathbf{k}_3 = \mathbf{f} + \mathbf{f}'\Delta + \frac{1}{2}\mathbf{f}''(\Delta, \Delta) + \frac{1}{6}\mathbf{f}'''(\Delta, \Delta, \Delta) + \mathcal{O}(\tau^4)
\]
Substituting $\Delta$ into each component term up to $O(\tau^3)$ yield:
\begin{itemize}
    \item $\mathbf{f}'\Delta = \tau \mathbf{f}' \mathbf{f} + \tau^2 \mathbf{f}' \mathbf{f}' \mathbf{f} + \frac{1}{4} \tau^3 \mathbf{f}' \mathbf{f}''(\mathbf{f}, \mathbf{f})$
    \item $\frac{1}{2} \mathbf{f}''(\Delta, \Delta) = \frac{1}{2} \tau^2 \mathbf{f}''( \mathbf{f}, \mathbf{f}) + \tau^3 \mathbf{f}''(\mathbf{f}'\mathbf{f}, \mathbf{f}) + \mathcal{O}(\tau^4)$
    \item $\frac{1}{6} \mathbf{f}'''(\Delta, \Delta, \Delta) = \frac{1}{6} \tau^3 \mathbf{f}'''\mathbf{f},\mathbf{f},\mathbf{f}) + \mathcal{O}(h^4)$
\end{itemize}
Combining these gives:
\[
\mathbf{k}_3 = \mathbf{f} + \tau f'f + \tau^2\left[\mathbf{f}' \mathbf{f}' \mathbf{f} + \frac{1}{2}\mathbf{f}''(\mathbf{f},\mathbf{f}) \right] + \tau^3 \left[\frac{1}{4} \mathbf{f}'\mathbf{f}''(\mathbf{f},\mathbf{f}) + \mathbf{f}''(\mathbf{f}'\mathbf{f}, \mathbf{f}) + \frac{1}{6} \mathbf{f}'''(\mathbf{f},\mathbf{f},\mathbf{f})\right] + \mathcal{O}(\tau^4)
\]

\noindent {\bf Combining the Stages}: Since $u_{n+1} = u_n + \tau \left(\frac{1}{6}k_1 + \frac{2}{3}k_2 + \frac{1}{6}k_3\right)$, collecting terms by powers of $h$:
\begin{itemize}
    \item \textbf{$\tau^1$ term:} $\left(\frac{1}{6} + \frac{2}{3} + \frac{1}{6}\right) \mathbf{f} = \mathbf{f}$
    \item \textbf{$\tau^2$ term:} $\left(\frac{2}{3}\cdot\frac{1}{2} + \frac{1}{6}\cdot 1\right) \mathbf{f}'\mathbf{f} = \frac{1}{2} \mathbf{f}'\mathbf{f}$
    \item \textbf{$\tau^3$ term:} $ \frac{2}{3}\left(\frac{1}{8} \mathbf{f}''(\mathbf{f},\mathbf{f})\right) + \frac{1}{6}\left((\mathbf{f}'\mathbf{f}'\mathbf{f} + \frac{1}{2}\mathbf{f}''(\mathbf{f},\mathbf{f})\right)  = \frac{1}{6}\mathbf{f}''(\mathbf{f},\mathbf{f}) + \frac{1}{6}\mathbf{f}'\mathbf{f}'\mathbf{f}$
    \item \textbf{$\tau^4$ term:} $ \frac{2}{3}\left(\frac{1}{48}\mathbf{f}'''(\mathbf{f},\mathbf{f},\mathbf{f})\right) + \frac{1}{6}\left(\frac{1}{4}\mathbf{f}'\mathbf{f}''(\mathbf{f},\mathbf{f}) + \mathbf{f}''(\mathbf{f}'\mathbf{f}, \mathbf{f}) + \frac{1}{6}\mathbf{f}'''(\mathbf{f},\mathbf{f},\mathbf{f})\right) $ \\
   \begin{equation} \label{RK3FourthOrderTerm}
  = \tau^4 \left[ \frac{1}{24}\mathbf{f}'''(\mathbf{f},\mathbf{f},\mathbf{f}) + \frac{1}{24}\mathbf{f}'\mathbf{f}''(\mathbf{f},\mathbf{f}) + \frac{1}{6}\mathbf{f}''(\mathbf{f}'\mathbf{f}, \mathbf{f}) \right]
\end{equation}
\end{itemize}

\noindent {\bf 3. The Local Truncation Error}

The local error $\xi_{n+1}$ is defined as the difference between the exact solution and the numerical step:
\[
\xi_{n+1} = \mathbf{\tilde{u}}(t_{n+1}) - \mathbf{u}_{n+1}
\]
Since the scheme matches perfectly up to $\mathcal{O}(\tau^3)$, the remaining leading-order error terms appear at $\mathcal{O}(\tau^4)$, scaled by the elementary differentials:

\begin{table}[h!]
\centering
\begin{tabular}{l@{\hspace{0.4cm}}c@{\hspace{0.4cm}}c@{\hspace{0.4cm}}c}
\toprule 
Elementary Differential & Exact Coef. & Numerical Coef. & Local Error Coef. \\
\midrule
\vspace{0.1in}
$T_1 = \mathbf{f}'''(\mathbf{f},\mathbf{f},\mathbf{f})$ & $\frac{1}{24}$ & $\frac{1}{24}$ & $\mathbf{0}$ \\ 
\vspace{0.1in}
$T_2 = \mathbf{f}''(\mathbf{f}'\mathbf{f}, \mathbf{f})$ & $\frac{3}{24}$ & $\frac{4}{24}$ & $\mathbf{-\frac{1}{24}}$ \\ 
\vspace{0.1in}
$T_3 = \mathbf{f}'\mathbf{f}''(\mathbf{f},\mathbf{f})$  & $\frac{1}{24}$ & $\frac{1}{24}$ & $\mathbf{0}$ \\ 
$T_4 = \mathbf{f}'\mathbf{f}'\mathbf{f}'\mathbf{f}$  & $\frac{1}{24}$ & $        0        $ & $\mathbf{\frac{1}{24}}$ \\
\bottomrule
\end{tabular}
\end{table}

\noindent {\bf Final Local Error Expression}

Combining the non-zero coefficients, the detailed local truncation error expression for the classical RK3 scheme is:
\begin{equation} \label{localError}
\xi_{n+1} = \tau^4 \left( -\frac{1}{24} \mathbf{f}''(\mathbf{f}'\mathbf{f}, \mathbf{f}) + \frac{1}{24} \mathbf{f}'\mathbf{f}'\mathbf{f}'\mathbf{f} \right) + \mathcal{O}(\tau^5)
\end{equation}
In the computation of this case study, we omit the $\mathcal{O}(\tau^5)$ and use the fourth order principal part as the local error estimate. This is to be able to calculate the contribution of the local error to the global error, while not causing a big error of error. We will show by how much we can make the calculated global error informative. In case the reliability of only using the fourth order terms is doubted, one can calculate the terms of the error in order 5 and omit $\mathcal{O}(\tau^6)$, but we don't want to make thing too tedious in the discussion.

\section{Error propagation in a time step}

The exact error propagation in the numerical smoothness framework will help us in the surprising computation of the global error. Although we will just compute the principal part of the global error, which is not a traditional worst case scenario error bound, the information delivered by the principal part error plots is surprisingly abundant.

We start this work in the general case. Let $\mathbf{u}(t)$ and $\mathbf{v}(t)$ be two solutions of $\mathbf{u}' = \mathbf{f}(\mathbf{u})$. They have $\mathbf{u}_0$ and $\mathbf{v}_0$ as their initial values respectively at $t=a$. Taking the difference of the two equations satisfied by $\mathbf{u}$ and $\mathbf{v}$, we get
$$
( \, \mathbf{u}(t)-\mathbf{v}(t) \, )' = \mathbf{f}(\mathbf{u}(t)) - \mathbf{f}(\mathbf{v}(t)).
$$
Let $g(\theta) =  \mathbf{f} [ \theta \mathbf{u}(t) + (1-\theta) \mathbf{v}(t))]$. Then $g(0) =  \mathbf{f}(\mathbf{v}(t))$, $g(1) = \mathbf{f}(\mathbf{u}(t))$. Moreover,
$$
g'(\theta) =  \mathbf{f}' [ \theta \, \mathbf{u}(t) + (1-\theta) \, \mathbf{v}(t) ]  [\mathbf{u}(t)-\mathbf{v}(t) ].
$$
Now,
\begin{align*}
(\mathbf{u}(t)-\mathbf{v}(t))' &=  \mathbf{f}(\mathbf{u}(t)) - \mathbf{f}(\mathbf{v}(t)) = g(1)-g(0) = \int_0^1 g'(\theta)\, d\theta \\
 &= \int_0^1 \mathbf{f}' [ \theta \, \mathbf{u}(t) + (1-\theta) \, \mathbf{v}(t) ] d\theta \,[\mathbf{u}(t)-\mathbf{v}(t) ] .
\end{align*}
It can be observed that the last equation can be taken as a linear homogeneous equation about $ \mathbf{u}(t)-\mathbf{v}(t)$, which can be solved in $[a,t]$ as
\begin{equation}  \label{Propag}
 \mathbf{u}(t)-\mathbf{v}(t) = e^{ \int_{a}^{t}  \int_0^1 \mathbf{f}' [ \theta \, \mathbf{u}(\chi) + (1-\theta) \, \mathbf{v}(\chi) ] d\theta \,d\chi }\,[\mathbf{u}(a)-\mathbf{v}(a) ].
\end{equation}
The matrix exponential on the right hand side is basically from the Jacobian matrix. It carries all the information and actions of error propagation. 

{\bf Remark}: The exponential  matrix in the RHS is very similar to the Monodromy matrix in the Floquet theory on the stability of the limit cycles of ODE systems. The difference is that they are focusing on one solution - the limit cycle, we are working on the distance between two solution of an ODE, to see how the distance evolves from $a$ to $t$. We will use the error estimate we have done for $ \mathbf{u}(a)-\mathbf{v}(a) $ to make sure it is already a close pair of initial points. Then we show that the distance remains small in a time step. We will always make $\mathbf{v}(t) = \mathbf{\tilde{u}}(t)$ and $\mathbf{u}(t)$ to be the original solution of our equation to solve.  

{\bf Remark}: Everything in the matrix exponential  depends on $\mathbf{u}(a)$, $\mathbf{v}(a)$, $\mathbf{u}(a)-\mathbf{v}(a)$, and the function $\mathbf{f}(\mathbf{u})$. The integral with the dummy variable $\theta$ needs to be further studied and specified to make the properties of the error propagation clear. Under different situations, we can use the last equation in deeper analysis of the exact error propagation  between two solutions in an interval $[a,t]$.

{\bf $\diamondsuit$ Estimating the error propagation in a time step}
Consider $\mathbf{u}(t)-\mathbf{\tilde{u}}(t)$, $t \in [t_n, t_{n+1}]$. This is the propagation error defined as in Splitting 2. To be clear, we are using (\ref{Propag}) with $a=t_n$, $t=t_{n+1}$, $\mathbf{u}(t)$ being the original solution, $\mathbf{v}(t) = \mathbf{\tilde{u}}(t)$, $\mathbf{v}(t_n) = \mathbf{u}_N(t_n)$ in the way the initial value of $\mathbf{\tilde{u}}(t_n)$ is defined in $[t_n,t_{n+1}]$, and $\mathbf{u}(t_n)$ is just the value of the original solution at $t_n$.

The ratio of the norms of $\mathbf{u}(t)-\mathbf{\tilde{u}}(t)$ and $\mathbf{u}(t_n)-\mathbf{u}_N(t_n)$ is the error propagation rate. In addition, we also need to keep the information on how the differential equation evolve $\mathbf{u}(t_n)-\mathbf{u}_N(t_n)$ by expansion, contraction, rotation and so on. The details need to be carefully estimated in each step, and composited between steps. If we fail to take the details into account, we will not make  very tight estimate of the error propagation rate step by step.

At $t_n$, $\tilde{u}(t_n)$ has been computed, it is reasonable to assume that $\tilde{u}(t_n)$ is bounded. In addition, $\mathbf{u}(t_n) - \mathbf{\tilde{u}}(t_n) $ has been estimated with an error bound of $\mathcal{O}(\tau^3)$ at the end of the last timestep. With the boundedness of $\mathbf{f}'$, in the very short time of $[t_n, t_{n+1}]$, we can get  $\mathbf{u}(t) - \mathbf{\tilde{u}}(t) = \mathcal{O}(\tau^3)$, thus
$$
\mathbf{f}' [ \theta \, \mathbf{u}(t) + (1-\theta) \, \mathbf{\tilde{u}}(t) ] 
=  \mathbf{f}' [  \mathbf{\tilde{u}}(t) + \theta \,( \mathbf{u}(t) -\, \mathbf{\tilde{u}}(t)) ]
= \mathbf{f}'( \mathbf{\tilde{u}}(t)) + \mathcal{O}(\tau^3).
$$
Consequently, by emphasizing $\mathbf{\tilde{u}}(t_n) = \mathbf{u}_N(t_n)$ in $[t_n,t_{n+1}]$, 
$$
 \int_{t_n}^{t_{n+1}}  \int_0^1 \mathbf{f}' [ \theta \, \mathbf{u}(\chi) + (1-\theta) \, \mathbf{\tilde{u}}(\chi) ] d\theta \,d\chi 
= \int_{t_n}^{t_{n+1}}  \mathbf{f}'( \mathbf{\tilde{u}}(\chi)) d\chi + \mathcal{O}(\tau^4) .
$$
Now one can also do the  expansion of $\mathbf{f}'(\mathbf{\tilde{u}}(\chi))$ at $\mathbf{u}_N(t_n)$ , while omitting the evaluation point, 
$$
\mathbf{f}'( \mathbf{\tilde{u}}(\chi)) = \mathbf{f}' + \mathbf{f}''\mathbf{f} (\chi-t_n) + \frac{1}{2}[\mathbf{f}'''(\mathbf{f},\mathbf{f}) + \mathbf{f}'' \mathbf{f}' \mathbf{f} ] (\chi-t_n)^2             + \mathcal{O}(\tau^3). 
$$ 
 After integrating $\chi \in [t_n, t_{n+1}]$ and moving all $\mathcal{O}(\tau^4)$ terms to the LHS, we get
$$
 \int_{t_n}^{t_{n+1}}  \mathbf{f}'( \mathbf{\tilde{u}}(\chi)) d\chi + \mathcal{O}(\tau^4)
= \tau \mathbf{f}' + \frac{\tau^2}{2} \mathbf{f}'' \mathbf{f} + \frac{\tau^3}{6} [\mathbf{f}'''(\mathbf{f},\mathbf{f}) + \mathbf{f}'' \mathbf{f}' \mathbf{f} ]
$$
with all the terms of the RHS evaluated at $\mathbf{u}_N(t_n)$. Realizing that the entire RHS is a matrix, let's temporally call the matrix by $M$. In this case we have the matrix exponential in (\ref{Propag}) equals to $e^{M + \mathcal{O}(\tau^4)}$. 

Here we use the Taylor expansion again on $e^M$. It is easy to see
$$
e^M = I + M + \frac{1}{2} M^2 + \frac{1}{6} M^3 +  \mathcal{O}(\tau^4).
$$ 
Let's define   
$$
P_n = I + M + \frac{1}{2} M^2 + \frac{1}{6} M^3 ,
$$
then we can rewrite (\ref{Propag}) as

\begin{equation} \label{PnPropag}
 \mathbf{u}(t_{n+1})-\mathbf{\tilde{u}}(t_{n+1}) = ( P_n  + \mathcal{O}(\tau^4)   ) \,[\mathbf{u}(t_n)-\mathbf{u}_N (t_n) ].
\end{equation}

We call $P_n$  the {\bf error propagation matrix} of $[t_n, t_{n+1}]$, because now we have approximately
$$
\mathbf{u}(t_{n+1})-\mathbf{\tilde{u}}(t_{n+1}) \approx P_n \, [\mathbf{u}(t_n)-\mathbf{u}_N (t_n) ].
$$
Here we use  the approximation in the sense of replacing $e^{\beta \tau + \mathcal{O}(\tau^4)}$ by $e^{\beta \tau}$ for the error propagation rate. By doing so we can avoid a lot of the tediousness of doing rigorous inequalities. Since $\mathbf{u}(t_n)-\mathbf{u}_N (t_n) = \mathcal{O}(\tau^3)$, we believe that the 
$\mathcal{O}(\tau^4)$ in (\ref{PnPropag}) must be sufficiently small, at least for a short time of consecutive composition of error propagation matrix actions in the next section. 


\section{Estimating error propagation over a short time}

By a short time, we mean one of a few different cases. This can be one period for a periodic solution, or a similar time interval. For nonlinear systems, it could happen that the contraction property will not show up unless a long enough time interval is allowed for an observer. A two level error propagation technique was given in \cite{tsEwing}. In this section, we discuss how to estimate the error propagation and accumulation together within a short time consisting of several time steps. The long time error estimation is left to be discussed in the following section, to taken advantage of the attractions, such as the limit cycle of a VdP system.

If we use $\xi_{n+1}$ as the local error for the time step, by Splitting 2,  we have 
\begin{align*}
\mathbf{u}(t_{n+1}) - \mathbf{u}_N(t_{n+1}) &= \mathbf{u}(t_{n+1}) - \mathbf{\tilde{u}}(t_{n+1}) +  \mathbf{\tilde{u}}(t_{n+1}) -  \mathbf{u}_N(t_{n+1}) \\
                                                                      & \approx P_n \, [\mathbf{u}(t_n)-\mathbf{u}_N (t_n) ] + \xi_{n+1}.
\end{align*}
Denoting the global error by $e_n = \mathbf{u}(t_n)-\mathbf{u}_N (t_n)$, we can rewrite the error relation above briefly, that is
\begin{equation} \label{globalErrorEndStep}
e_{n+1} \approx P_n e_n + \xi_{n+1}.
\end{equation}
Using this error relation repeatedly, we get
\begin{equation} \label{propagation}
e_{n+1} \approx \xi_{n+1} + P_n \xi_n + P_n P_{n-1} \xi_{n-1} + \cdots + P_n \cdots P_1 \xi_1 + P_n \cdots P_1 P_0 e_0.
\end{equation}
In this way, we get to optimize the error propagation by calculating how the global error is propagated step by step. The error propagation matrices $P_n$ are easily computable. More importantly, while losing some rigorousness with the thrown away $\mathcal{O}(\tau^4)$, we get to keep the multidimensional behavior of the error propagation matrix $P_n$. 

For the local error $\xi_{n+1}$, we can compute its main part of $\mathcal{O}(\tau^4)$ while neglecting the remaining part of $\mathcal{O}(\tau^5)$. For the global error, one can easily compute the products $\Pi_{i=k}^n P_i$ and observe that all these products for $k=0, \cdots, n$ are uniformly bounded. Then, from (\ref{propagation}), it is obvious $\| e_{n+1} \| = \mathcal{O}(\tau^3)$.

{\bf $\diamondsuit$ The numerical experiment}

In Figure \ref{figVdP}, we present the results of our RK3 solution for the VdP system with $\mu = 3$. Let specify the meaning of the six plots one by one.

\begin{figure}[htbp]
    \centering
    \includegraphics[width=\textwidth]{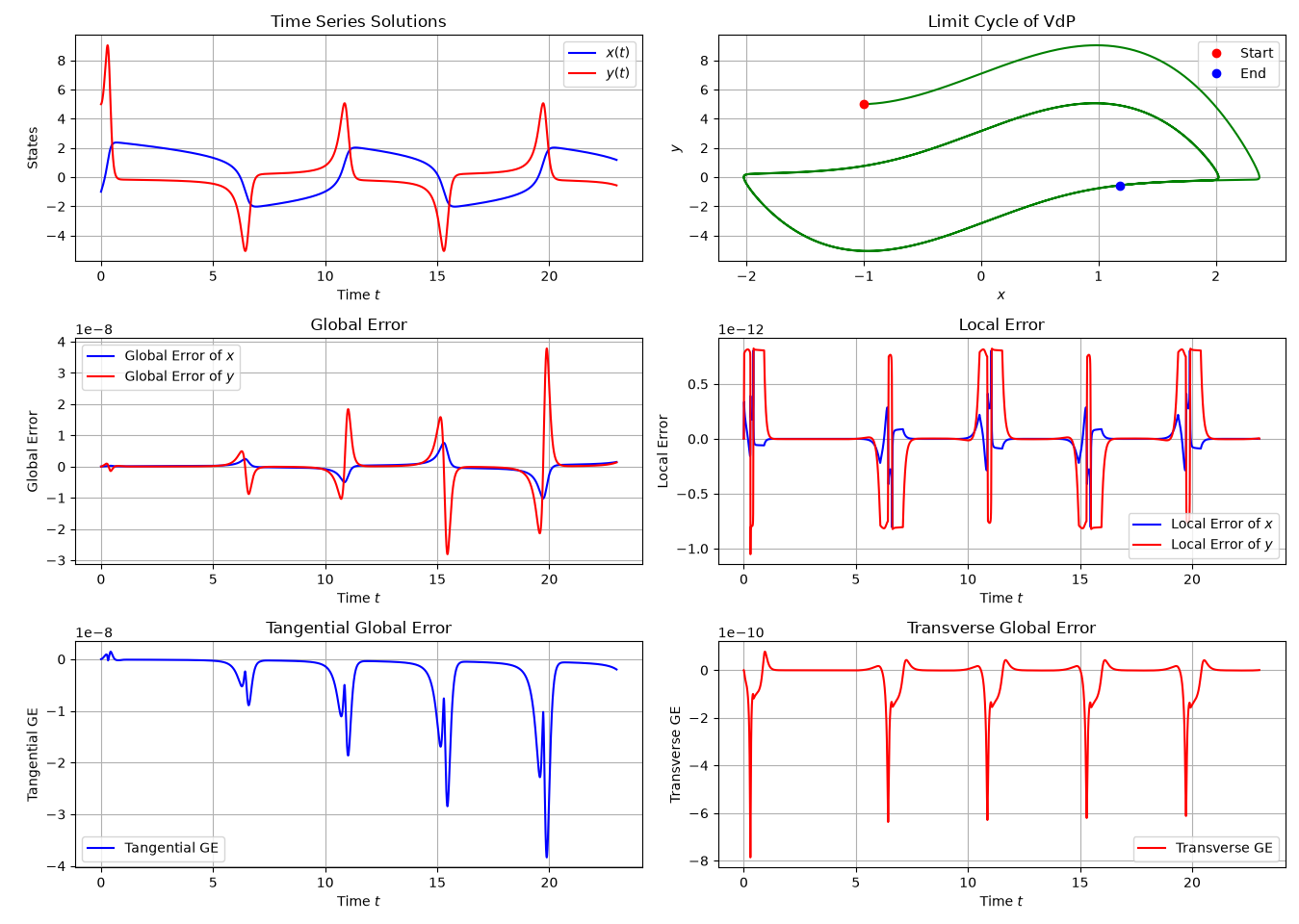}
    \caption{Error of Van der Pol solution}
    \label{figVdP}
\end{figure}

We choose $\mu = 3$ because this makes the system significantly nonlinear, slightly stiff, merging into the limit cycle quickly in a plot, presenting the effectiveness of our error propagation  matrix clearly, and helping us explaining the advantages of the numerical smoothness framework for error analysis explicitly. We will focus on this one numerical experiment in this lecture note. We will not touch an extremely stiff system ($\mu >>1$), not touch another numerical scheme, which would increase the length of the chapter but not making a big difference for us in our effort of promoting the numerical smoothness paradigm. 

{\bf Upper Left Plot:} The upper left plot shows the two components of the numerical solution of $x(t)$ and $y(t)$. One can see that the solution moves through the the two parts of the orbit where $y(t) \approx 0$ with most of the time. In comparison, the solution moves quickly along the part of the orbit where $y(t)$ is positive and away from $0$ and along the part of the orbit where $y(t)$ is negative and away from $0$. Where $y(t)>0$, we see a positive sharp spike of $y(t)$ and very fast growth of $x(t)$ toward $2.0$. Where $y(t)<0$, we see a negative sharp spike of $y(t)$ and very fast decrease of $x(t)$ toward $-2.0$. All the spikes in the other plots also appear at the same time. These quick exponential variations are due to the nonlinear short time energization and dissipation of the VdP model. They are more severe with bigger $\mu$, but these spikes are not considered as a consequence of stiffness.

{\bf Upper Right Plot:} The upper right plot is the limit cycle revealed by the solution orbit when $t$ is sufficiently large, and the initial part of the orbit before the solution merges into the limit cycle. When the initial point is set to be away from the  limit cycle, the solution will move almost parallel to the limit cycle until it enters a region of the phase plane above or below the line $y=0$ outside of $x \in [-1,1]$. From one of these regions, the solution will be captured by the strong attraction of the limit cycle or the outward extension of the limit cycle near $y=0$. One thing might be misleading for us, which is the smooth look of the orbit where $y(t)$ is away from $0$. In these two parts of the orbit, $x(t)$ and $y(t)$ are accelerating/decelerating very quickly and simultaneously. This is where adaptive error control actions are needed.

{\bf Middle Right Plot:} The middle right plot is the local truncation error $\xi_n$, $n = 0, 1, \cdots $. One can recognize the effect of the adaptive local error control. For each time step, the step size $\tau$ is chosen so that the local error $\xi_{n+1}$ computed by (\ref{localError}) is under $10^{-12}$. The 4th order derivatives in (\ref{localError}) are available to us from the smoothness indicator at each $t_n$. The plot shows that the local error stays much smaller than the error tolerance most of the time. Only the spikes of the local error need to be dealt with twice in a period of the limit cycle. 

{\bf Middle Left Plot:} This plot shows the global error $e_n$ in all the time steps. The global error $e_n$ for all n are computed by (\ref{propagation}). The global error terms are showing fast growth and fast decay where $y(t)$ is away from $0$. In the parts of the orbit where $y(t) \approx 0$, the global error shows to be almost $0$, because it is under lasting compression over there while $u(t)$ is moving forward very slowly. However, the global error is surviving during the compression and its residual is actually bigger than the residual of the global error during the last compression a half of the limit cycle period ago. When the next energization comes, the residual of the global error grows exponentially to a higher peak from the bigger residual. 

{\bf Lower Left Plot:} This plot shows the component of the global error in the tangential direction of the orbit. At each $t_n$, we project $e_n$ to the tangent direction of the orbit at $\mathbf{u}_N(t_n)$. The unit tangent vector of the orbit at $t_n$ is $[f_1, f_2] = \frac{1}{\| \mathbf{f} \|}  \mathbf{f}$ evaluated at $\mathbf{u}_N(t_n)$. Then the tangential component of $e_n$ is $e_n^T = [f_1,f_2] \cdot e_n$. This plot shows all the $e_n^T$ as a curve. It is very interesting to see that the profile of the spikes of the tangential error curve lines up purely in a straight line. This matches the theoretical expectation that the error along the orbit has no self correction mechanism, and there is a simple accumulation over time, regardless of the local exponential growth and decay from the energizations and dissipations.  

{\bf Lower Right Plot:} This is the plot of the transverse component of the global error perpendicular to tangential component. At each $t_n$, it is computed by $e_n^P = [f_2, -f_1] \cdot e_n$. In the plot of the $e_n^P$ curve, we can see that the profile of the spikes form a horizontal straight line. So there is no growth whatsoever. The accumulation of local error and the error compression in the transverse direction have reached a balance, keeping $e_n^P$ staying below $10^{-10}$. In comparison, $e_n^T$ is in the magnitude $10^{-8}$ and linearly growing.

In addition to Figure \ref{figVdP}, we also computed the period $T$ of the limit cycle, the matrix $Q^T$ consisting of the consecutive product of the error propagation matrices over the last period $[23-T, 23]$, and the eigenvalues $\{ \nu_1, \nu_2 \}$ and eigenvectors of $Q^T$. The results show that
$$
T \approx 8.859095, \qquad \nu_2 \approx 1, \qquad \nu_1 \approx  10^{-16}. 
$$
The first eigenvalue $\nu_2 =1$ means that, over a period $T$, the tangential error should be unchanged by the pure error propagation over one period.  The second eigenvalue $\nu_1 \approx 0$ tells us that the transverse error is compressed by the strong attraction of the limit cycle. Both of these results are consistent with the Lyapunov exponents in the Floquet theory: $\nu_1 = e^{\lambda_1} \approx 0$ and $\nu_2 = e^{\lambda_2} \approx 1$. See the next section for a little more details about the Floquet theory.

\section {Long term error propagation along the limit cycle}

Most part of this section is a modified copy of a Gemini AI article on the topic. We obtained their permission to include the article in this lecture note. 

To estimate the attraction rate (or contraction rate) of the Van der Pol limit cycle over a single period, we look to the {\bf Lyapunov exponents} or, more specifically, the {\bf integrated divergence} of the system's vector field along the periodic orbit.
By applying Floquet theory or calculating the time-averaged divergence, you can determine how quickly nearby trajectories converge to the limit cycle.
Here is the analytical and numerical framework to estimate this rate.

{\bf 1. The Governing System}

The classic Van der Pol oscillator is governed by the second-order differential equation: 
$$
\ddot{x} - \mu (1-x^2) \dot{x} + x =0 .
$$
Rewriting this as a first-order system of autonomous differential equations by making $x=x$ and $y=\dot{x}$:
\begin{align*}
\dot{x} &= y, \\
\dot{y} &= \mu (1-x^2) y - x,
\end{align*}
where $\mu > 0$ is the nonlinearity parameter controlling the damping. If we make $\mathbf{x} = (x, y)^T$, we get the VdP system in the vector form 
$$
\dot{\mathbf{u}} = \mathbf{f}(\mathbf{u}).
$$

{\bf 2. The Theoretical Framework: Divergence and Phase Space Volume}

The local contraction or expansion rate of a phase space volume element along a trajectory is given by the divergence of the vector field, $\nabla \cdot \mathbf{f}$.
For the Van der Pol system, the Jacobian matrix $J(\mathbf{u})$ is:
$$J(x, y) = \begin{pmatrix} 0 & 1 \\ -2\mu x y - 1 & \mu(1 - x^2) \end{pmatrix}$$
The trace of the Jacobian matrix yields the divergence:
$$\text{Tr}(J) = \nabla \cdot \mathbf{f} = \mu(1 - x^2)$$
•	When $|x| > 1$, $\nabla \cdot \mathbf{f} < 0$ (the flow is locally contracting/attracting).\\
•	When $|x| < 1$, $\nabla \cdot \mathbf{f} > 0$ (the flow is locally expanding/repelling).

{\bf 3. Estimating the Net Attraction Rate Over a Period ($T$)}

Because a limit cycle is a closed isolated periodic orbit $\gamma(t)$ with period $T$, a trajectory spending time in both the repelling and contracting regions must balance out to be net attractive overall in a period.
According to Liouville's formula, the growth or decay of an infinitesimal area element $\delta A(t)$ transverse to the orbit over one full period is:
$$\delta A(T) = \delta A(0) \exp \left( \int_0^T \nabla \cdot \mathbf{f}(\gamma(t)) \, dt \right)$$
Therefore, the {\bf average characteristic exponent} (which corresponds to the non-zero Floquet exponent, $\lambda_1$) defines the net attraction rate:
$$\lambda_1 = \frac{1}{T} \int_0^T \mu \left(1 - x(t)^2\right) dt$$
Since one Characteristic/Floquet exponent associated with the phase along the continuous limit cycle is always exactly $\lambda_2 = 0$, $\lambda_1$ represents the transverse stability. For an attractive limit cycle, $\lambda_1$ must be negative.

{\bf 4. Estimation Methods based on $\mu$}

{\bf Case A}: Weakly Nonlinear Regime ($\mu \ll 1$)

When $\mu$ is very small, the limit cycle is nearly a circle of radius $2$, described roughly by $x(t) \approx 2\cos(t)$ with period $T \approx 2\pi$.
Substituting this approximation into the attraction rate integral:
$$\lambda_1 \approx \frac{1}{2\pi} \int_0^{2\pi} \mu \left(1 - 4\cos^2(t)\right) dt$$
Using the identity $\cos^2(t) = \frac{1 + \cos(2t)}{2}$:
$$\lambda_1 \approx \frac{\mu}{2\pi} \int_0^{2\pi} (1 - 2 - 2\cos(2t)) \, dt = \frac{\mu}{2\pi} \left[ -t \right]_0^{2\pi} = -\mu$$
•	Result: For small $\mu$, the net attraction rate per unit time is approximately $-\mu$. The orbit contracts by a factor of $e^{-2\pi\mu}$ per period.

{\bf Case B: Strongly Nonlinear / Relaxation Regime ($\mu \gg 1$)}

When $\mu$ is large, the system becomes a relaxation oscillator. The trajectory spends most of its period $T$ creeping slowly along the attracting branches of the cubic slow manifold ($y \approx \frac{x^3}{3} - x$ in Liénard coordinates), where $|x| > 1$. The rapid horizontal "jumps" between branches take negligible time.
In this regime, the period scales as:
$$T \approx (3 - 2\ln 2)\mu \approx 1.614\mu$$
Because the trajectory spends almost the entire period bound tightly to the slow manifold where $x$ transitions from $2$ down to $1$, the integrated divergence over these slow segments dominates. Asymptotically, the time-averaged attraction rate scales directly with the geometry of this manifold, yielding a highly stable, strongly attracting rate where $\lambda_1 \sim -\mathcal{O}(\mu^2)$ during the slow phases, balancing out to a net average exponent that scales as $-\mathcal{O}(1)$ or structural fractions of $\mu$ depending on the specific coordinate scaling chosen (e.g., Liénard vs. classic physical coordinates).

{\bf 5. Numerical Estimation Pipeline}

If you are working with an arbitrary $\mu$, you can compute the attraction rate accurately in Python using these steps:
\begin{enumerate}
\item	Find the Limit Cycle: Integrate the system forward long enough for transient behaviors to die out until it settles onto $\gamma(t)$.\\
\item	Determine Period $T$: Use root-finding or event detection (e.g., SciPy's solve-ivp with an event looking for directional zero-crossings of $x$) to find the exact period $T$.\\
\item	Evaluate the Integral: Concurrently integrate the system state and the variational divergence equation:
\end{enumerate}
$$\dot{z} = \mu(1 - x^2)$$
The value $z(T) / T$ at the end of one cycle yields your exact attraction rate $\lambda_1$.

The Gemini article ended here. Instead of the $\lambda_1 = z(T)/T$, we actually need $e^{\lambda_1 T} $, which is the attraction rate of the limit cycle over a period $T$. Alternatively, we can also  compute $Q_T = P_{m-1}P_{m-2} \cdots P_1 P_0$ , where each one of the error propagation matrices, $P_j$, comes from the numerical solution of the VdP for exactly one step, $Q^T$ is the product of them over a period $T$ with an initial point almost on the limit cycle. We can understand the matrix $Q_T$ as an error propagation matrix of limit cycle, except it depends on the choice of the initial/ending point. However, the eigenvalues of the matrix should be an invariant of the limit cycle. Our computation shows, indeed as expected, one eigenvalue of $Q_T$ is $1$, corresponding to the error in the tangential direction of the orbit. Meanwhile, the other eigenvalue of $Q_T$ is positive but extremely close to $0$, something like $10^{-15}$, which is consistent to the $e^{\lambda_1 T}$. For example, when $\mu$ is moderately large, like $\mu = 3$, we numerically computed $T \approx 8.859095$, $\lambda_1 \approx -3.94$, making $e^{\lambda_1 T} \approx e^{-34.9} \approx 6.94* 10^{-16}$. This is to say that the limit cycle possesses an extremely strong error contraction ability in its transverse direction. Finally, as the error contraction factor of the limit cycle, to be safe, we define 
$$
\Theta =  e^{\lambda_1 T}.
$$

If we solve the VdP in a time interval $[0,S]$ and we really want to see the error between $\mathbf{u}(S)$ and $\mathbf{u}_N (S)$, then we have to use the  estimated error  given by (\ref{propagation}) and keep in mind that the entire $\{e_n\}$ is basically the error along the orbit, while the transverse component of the error is pretty much negligible.

If we continue the numerical solution past the time interval $[0,S]$, it is easy to see that the numerical solution will continue to follow the limit cycle within a fixed average distance from the limit cycle. However, the numerical solution's phase position in the limit cycle can shift. The shift could be forward or backward depending on the numerical scheme and its local error's tendency.  If the phase position is important, one may need to study of the local error in more detail and make some corrections accordingly.

\section{Case study conclusions}

We did this case study to display the advantages of the numerical smoothness framework for error estimation. 

\begin{itemize}

\item The error analysis remains the same for different single step numerical schemes, except it needs the formula of the used scheme's local error. The smoothness indicator is computed at each $t_n$ using the same set of formulas. 

\item  The error propagation analysis in the framework is on a pair of two ODE solutions, which are close to each other with an distance proportional to the global error. In the case of VdP, we chose to employ the error propagation matrix, which is made possible by the numerical smoothness framework. we see how much information we are able to collect from the error propagation matrices. 

\item For this nonlinear VdP system, if we were to use a Lipschitz condition or the Gronwall Lemma in the error propagation (stability) analysis as in Splitting 1, our error estimate would have been bigger than the diameter of the limit cycle before $t=1$, if we used a reasonable time step size. 

\item Another point about the error propagation matrix is that this is the first time the author came up with the idea of using the technique learned from the  Floquet theory. It might have been used in the literature of numerical ODEs, but the author of this lecture note hasn't seen others doing so. It seems to be hard to use this idea for numerical error propagation (error propagation of schemes). Meanwhile, it should be very natural to use the technique to other ODE systems. It is amazing to be able to preserve so abundant information about the global error. Further research along the route can be expected.  
  
\item In Splitting 1, the ODE's high order derivative $u^{(k)}(t_n)$ is not computable except for at $t_n = t_0$, thus the local error is not computable. The numerically propagated error is also not computable. In the case study of VdP, we see that, thanks to Splitting 2, we can compute both. The local error $\xi_{n+1}$ can be computed by (\ref{localError}). The propagated error can be computed by (\ref{propagation}). 

\end{itemize}

\chapter{Further reading and future research}

\section{Further reading on existing results}

The author Tong Sun of this lecture note has written several papers on the topic of numerical smoothness. 

Two of the papers are on ODEs. Two are on the necessity of numerical smoothness. Three are on hyperbolic conservation laws. The other two are on parabolic PDEs. 

{\bf ODE results}

Back in 1998, R. Ewing  and T.Sun wrote the first paper on ODE \cite{tsEwing}, where they focused on the error propagation of the differential equation. 
At that time, the authors had not realized the difference between numerical solution's stability and smoothness, hence they defined an indicator which is exactly what we have today as the temporal smoothness indicator but they named it stability indicator. 10 years later, after writing two papers on numerical parabolic equations and their error analysis , Tong Sun finally realized that, under the error Splitting 2, the high order derivatives computed for $\tilde{u}_N(t)$ at each $t_n$ is in fact delivering the information on the smoothness of the numerical solution, not information on their stability. With this new understanding, it became clear that $u_N(t_n)$ should be smooth in some sense, and $u_N(t_n)$ should have been made smooth by the numerical scheme.  All of these together motivated the second ODE paper \cite{sunRK} by T. Sun, in which the traditional stability condition on Runge-Kutta methods and the stability region were modified to get the numerical smoothing condition on Runge-Kutta methods.

{\bf Necessity results}

The papers \cite{necessity, chou, chousun}  on the necessity of numerical smoothness have been covered in Section 2.3. So we focus on the papers on PDEs now.

{\bf Error splitting for PDEs}

An important difference between Splitting 1 and Splitting 2 for PDEs is the different ways  of using semi-discrete solutions. The plot below is how to do the error splitting in $[t_n,t_{n+1}]$ in the numerical smoothness framework (Splitting 2) for evolution equations. In the splitting, $u(t_n)$ is the value of the exact solution $u(t)$ at $t_n$, $u^c_n$ is the value of the computed numerical solution at $t_n$, $u(t_{n+1})$ is the value of the exact solution at $t_{n+1}$, $u^c_{n+1}$ is the value of the computed numerical solution at $t_{n+1}$. The global error $u(t_{n+1})-u^c_{n+1}$ at $t_{n+1}$ is split into three parts by inserting $\tilde{u}(t_{n+1})$ and $u^h(t_{n+1})$, where $\tilde{u}(t)$ is the solution of the original evolution equation in the time domain $[t_n,t_{n+1}]$ with initial condition $\tilde{u}(t_n) = u^c_n$ , and $u^h(t)$ is the semi-discrete solution in the time domain $[t_n,t_{n+1}]$ with initial condition $u^h(t_n) = u^c_n$.  

\setlength{\unitlength}{1.0mm}
\begin{picture}(70,30)(-10,10)
    \linethickness{0.6pt}
    \qbezier(30,8)(50,14)(70,26)
    \linethickness{0.5pt}
    \qbezier(30,0)(50, 6)(70,18)
    \qbezier(30,0)(50, 6)(70,15)
    \qbezier(30,0)(50, 6)(70,12)
    \put(25,-2){\makebox(0,0)[b]{$u^c_n$}}
    \put(24, 6){\makebox(0,0)[b]{$u(t_n)$}}
    \put(78, 9){\makebox(0,0)[b]{$u^c_{n+1}$}}
    \put(80.4,13.5){\makebox(0,0)[b]{$u^{h}(t_{n+1})$}}
    \put(80,19){\makebox(0,0)[b]{$\tilde{u}(t_{n+1})$}}
    \put(80,25){\makebox(0,0)[b]{$u(t_{n+1})$}}
    \put(5, -10){\vector(1,0){85}}
    \put(10,-11){\vector(0,1){45}}
    \put(29.8,-10){\line(0,1){18}}
    \put(70,-10){\line(0,1){36}}
    \put(30,-14){\makebox(0,0)[b]{$t_n$}}
    \put(70,-14){\makebox(0,0)[b]{$t_{n+1}$}}
\end{picture}

\vspace{1.0in}

In this three part splitting, $u(t_{n+1}) - \tilde{u}(t_{n+1})$ is the propagation is the global error $u(t_n) - u^c_n$ in the time step by the solutions of the evolution equation. $\tilde{u}(t_{n+1}) - u^h(t_{n+1})$ is the local spatial discretization error, $u^h(t_{n+1}) - u^c_{n+1}$ is the local temporal discretization error. One can see from the papers by Tong Sun and his collaborators how such a three part splitting can help the error analysis to optimize both error propagation estimates and local error estimates.

{\bf Parabolic PDE results}

The two papers on parabolic equations are \cite{sun} on partially implicit schemes and \cite{sunF} on semi-linear problems. \cite{sun} focuses the numerical smoothness issue when the scheme is not fully implicit (like backward Euler) but partially implicit. The smoothness indicator provides a reliable criterion at the run time to inform the user if local error is well under control. \cite{sunF} focuses on the nonlinearity of of the reaction terms in the equation. Taking advantage of the fact that the nonlinear reaction terms only acts on the low frequency components, we know that the numerical smoothness is expected to be valid. So there should be no trouble with local error. One want to focus on the error propagation. In the case of semi-discrete parabolic equations of \cite{sunF}, one get to use the one-sided Lipschitz condition to estimate short term error propagation and use the new concept of moving indicators defined in \cite{sunF} to do long term error propagation. Most importantly, the spatial smoothness indicator as in \cite{necessity} had not been found at that time (~2006). What served the role of a spatial smoothness indicator was the elliptic reconstruction. 

{\bf 1-D scalar conservation law results} 

The major progress of the numerical smoothness framework on the PDEs was made in the nonlinear scalar conservation lows. T. Sun and his PhD students wrote three papers on discontinuous Galerkin finite element method \cite{Rumsey, Fode} and WENO finite volume method \cite{WENO}. In addition to the cases of smooth solutions, T. Sun and Adamou Fode wrote a paper in the case of a shock wave with discontinuous Galerkin finite elements coupled with Raikine-Hugoniot front-tracking \cite{Fode}. For all these numerical hyperbolic PDE solutions, the major new issue is how to use the smoothness of the piecewise polynomial numerical solution hiding behind the derivative jumps. A set of auxiliary functions and various kinds of solutions were used to make the numerical smoothness work for local error estimates. The details are very tedious, but it is guaranteed to work because all the details involve only local analysis. After successfully overcoming all these  difficulties, the good news is that the error propagation gets to be done by the $L_1$-contraction property of the scalar conservation law for a pair of entropy solutions. 

Since the numerical smoothness concept is new and only Tong Sun and his colleagues and students have been working on this, the available results are limited to those mention above.

\section{Future research} 

As we have summarized, the research on the topic of numerical smoothness is extremely limited. Comparing to all the {\it a priori} error estimates along the   error analysis framework of numerical stability (Splitting 1), it is extremely unbalanced and strange to see that the numerical smoothness framework (Splitting 2) is basically left untouched for the past dozens of years. All the advantages of Splitting 2 have been blocked from people's attention by the ``nonexistence'' of the smoothness of $\tilde{u}(t)$ and the dominating belief of Lax Equivalence Theorem. 

When people talk about future work in a research area, they usually make a list of a few things they have in mind. In our case here, we have a list of things done in the framework of numerical smoothness in the last section. Having mentioned that, the rest of the numerical solutions of evolution equations can be studied for their error estimation in the numerical smoothness framework. A few examples of open topics are listed below.

For ODEs, one topic worth to be studied is the error propagation matrix which was introduced in Section 3.3 and 3.4 for the case study of the Van der Pol equation. This kind of error propagation matrix might work for systems of ODEs in general.

For parabolic equations, the spatial smoothness indicators based on the derivative jumps has bot been successfully used to parabolic equations. This may involve some form of discrete Green's functions to build bridges between $\tilde{u}(t)$ and $u^h(t)$ in local spatial error estimation. Todd Romutis has done some research on this in his PhD dissertation in 2018, which is available in OhioLink. The work is limited to 1-D linear parabolic equations.

For first order hyperbolic problems, research has only been done for 1-D scalar equations. 2-D scalar equations and hyperbolic systems are still to be studied. Second order hyperbolic problems still need to be studied. 

The programming of 2D code to implement front tracking of shocks, contact discontinuities and rarefaction waves are considered as hard to do. Now, one can do programming with AI. It takes researchers to work interactively with LLM to describe what are to be implemented for the specific scheme.

\end{document}